\documentclass{article}
\usepackage{times}
\usepackage[hyperindex=true,pageanchor=true,hyperfigures=true,backref=false]{hyperref} 
\usepackage{amsmath}
\usepackage{amssymb}
\usepackage{amsthm}
\usepackage[capitalise,nameinlink]{cleveref}
\usepackage{mathtools}
\usepackage{tikz-cd}
\usepackage[numbers]{natbib}
\usepackage{float}

\usepackage[table]{xcolor}

\usepackage{enumitem}
\usepackage{url}
\usepackage{dsfont} 
\DeclareSymbolFontAlphabet{\Bbb}{AMSb}
\usepackage[T1]{fontenc}

\usepackage{mleftright,xparse}

\usepackage{thm-restate,amsthm}
\newtheorem*{assumptionA}{Assumption A}
\newtheorem*{assumptionM}{Assumption M}  

\newlength{\myleftmargin}
\usepackage{amsmath}
\usepackage{amssymb}
\usepackage{amsthm}
\usepackage{color}
\DeclareSymbolFontAlphabet{\Bbb}{AMSb}

\newtheorem{theorem}{Theorem}[section]
\newtheorem{lemma}[theorem]{Lemma}

\newtheorem{corollary}[theorem]{Corollary}
\newtheorem{definition}[theorem]{Definition}

\newtheorem{example}[theorem]{Example}
\newtheorem{remark}[theorem]{Remark}

\newcommand{\atob}[2]{\emph{#1} $\Rightarrow$ \emph{#2}.} 
 
\newcommand{\ada}[1]{\emph{#1)\,}}

\newlength{\fixboxwidth}
\newcommand{\mycdot}{\,\cdot\,}

\newcommand{\N}{\mathbb{N}}

\newcommand{\R}{\mathbb{R}}

\renewcommand{\a}{\alpha}

\newcommand{\lb}{\lambda}

\newcommand{\s}{\sigma}

\newcommand{\om}{\omega}
\newcommand{\Om}{\Omega}

\newcommand{\snorm}[1]{\Vert #1 \Vert}

\newcommand{\inorm}[1]{\Vert #1 \Vert_\infty}

\title{Convergence Analysis of a Greedy Algorithm for Conditioning Gaussian Random Variables}

\author{Daniel Winkle, Ingo Steinwart, and Bernard Haasdonk\\
University of Stuttgart\\
\small Faculty 8: Mathematics and Physics\\
\small Stuttgart, Germany \\
\texttt{\small daniel.winkle@lut.fi}\\
\texttt{\small ingo.steinwart@mathematik.uni-stuttgart.de}\\
\texttt{\small bernard.haasdonk@mathematik.uni-stuttgart.de}
}

\begin{document}

\maketitle

\begin{abstract}
In the context of Gaussian conditioning, greedy algorithms iteratively select the most informative measurements, given an observed Gaussian random variable. However, the convergence analysis for conditioning
Gaussian random variables remains an open problem. We address this by introducing an operator $M$ that allows
us to transfer convergence rates of the observed Gaussian random variable approximation onto the conditional
Gaussian random variable. Furthermore, we apply greedy methods from approximation theory to obtain convergence rates. These greedy methods have already demonstrated optimal convergence rates within the setting
of kernel-based function approximation. In this paper, we establish an upper bound on the convergence rates
concerning the norm of the approximation error of the conditional covariance operator.
\end{abstract}



\section{Introduction}\label{sec:intro}
Conditioning of Gaussian random variables arises naturally in many settings, such as computer experiments \cite{santner2003design},
machine learning \cite{rasmussen2006gaussian}, and geostatistics \cite{stein1999interpolation}. In many of these contexts, particularly those involving functional data \cite{ramsay2005functional} or spatial fields \cite{cressie2015statistics}, it is essential to consider Gaussian random variables in infinite-dimensional settings typically modeled as Gaussian measures on separable Hilbert or Banach spaces. This infinite-dimensional
perspective is not only mathematically rich but also crucial for rigorous analysis and implementation in practical problems involving partial observations, interpolation, and uncertainty quantification. Noteworthy is the
infinite-dimensional case for Physics-informed Gaussian processes as shown in \cite{yang2018physics}. There the Gaussian random variable has to satisfy some physical constraints, thus is usually expressed by solving a partial differential
equation.

Moreover, there exists a close connection between kernel methods and Gaussian random variables in the
sense that a kernel induces a Gaussian random variable and conversely, a Gaussian random variable induces a
kernel, see \cite{kanagawa2018gaussian}. This becomes particularly clear by considering conditioning a Gaussian random variable $X$ on finitely many points, which corresponds to interpolation with the associated kernel.

In this work, we extend this idea by conditioning X on infinitely many points. In the kernel setting this
corresponds to restricting functions in the reproducing kernel Hilbert space (RKHS), for which a natural extension
exists \cite[Theorem 10.46]{wendland2004scattered}. 

Despite its importance, conditioning in infinite-dimensional settings is considerably less well understood
than in the finite-dimensional case. Nevertheless, several approaches have been developed. For Gaussian random variables taking values in separable Hilbert spaces one approach to solve the conditioning problem is using
shorted operators, see \cite{owhadi2018conditioning}. For separable Banach spaces \cite{tarieladze2007disintegration, travelletti2024disintegration} show that the conditioned Gaussian random variable can be represented by a series. More recently, \cite{steinwart2024conditioning} addressed the conditioning problem via regular
conditional probability by representing the conditioned Gaussian random variable as the limit of a sequence
converging in the Banach space norm. 

A common feature of these existing results is that they establish convergence but do not provide any rate
of convergence. The aim of this work is to close that gap. Specifically, given two jointly Gaussian random
variables $X$ and $Y$, we consider a sequence of $n \in \N$ linear measurements on $Y$, denoted as $\{f_j'(Y) \}_{j=0}^{n-1}$. We then derive explicit convergence rates for the operator norm of the conditional covariance operator 
\begin{align}\label{eq:main}
\Vert \textup{cov}(X \vert Y) - \textup{cov} \left( X \vert \{f_j'(Y)\}_{j=0}^{n-1} \right) \Vert \leq C a_n
\end{align} 
for some sequence $(a_n)$ and $C>0$, see Theorem \ref{Theorem:main convergence rate generalized} and Corollary \ref{Corollary:Polynomial Convergence Rate P-Greedy} for the precise statements.  In particular,
when $X$ is a Gaussian process with continuous paths, Equation \eqref{eq:main} implies a convergence rate for the associated
covariance function of $X\vert Y$ in the supremum norm \cite{steinwart2024conditioning}. 

We obtain \eqref{eq:main} by introducing a bounded conditioning operator $M$ such that the observed Gaussian random
variable $Y$ is mapped to the conditional expectation $Z := \mathbb{E}(X \vert Y )$, i.e., $MY = Z$. 
This operator enables the transfer of known posterior covariance convergence rates of 
$\Vert \textup{cov}(Y \vert \{ f_j'(Y) \}_{j=0}^{n-1})\Vert$
 \cite{santin2017convergence, wendland2004scattered} onto the left-hand side of \eqref{eq:main}. The boundedness of the operator $M$ is crucial for this transfer, making the choice of Banach spaces for $X$ and $Y$ a key consideration, even though these spaces do not directly affect the Gaussian variables \cite{stroock2010probability}.
 
We also emphasize that our analysis is not limited to point evaluations, which is common in the literature
\cite{vanderVaart2011information}, but instead we allow any general linear measurement as in \cite{travelletti2024disintegration}. 

To select the measurements $\{f_j'(Y) \}_{j=0}^{n-1}$ we suggest an iterative algorithm that has been well studied in the
kernel setting, the so-called $P$-greedy algorithm \cite{santin2017convergence}. In simple terms, we select the measurement that provides
the ‘most’ additional information, taking into account all the prior measurements. This approach is the same as
in Bayesian optimization, when one maximizes the covariance function, see e.g. \cite[Chapter 8.2]{rasmussen2006gaussian}. 

Greedy algorithms are known to provide fast convergence rates for kernel methods \cite{wenzel2022greedy}. Now, as already
mentioned kernel methods and Gaussian random variables are closely related \cite{kanagawa2018gaussian}. Therefore, it is natural to apply greedy algorithms to Gaussian random variables as well \cite{smola2000sparse}. However, a rigorous analysis of convergence rates in specific norms remains an open problem, which we address in this paper.

Our analysis thus bridges a theoretical gap and provides a practical tool for uncertainty quantification in
high-dimensional or functional data models. 

We also want to emphasize that the focus lies on Gaussian random variables, even though it might be possible to extend some of the analysis onto student-$t$ processes, see \cite{shah2014student, tracey2018upgrading, jylanki2011robust}.

This paper is structured as follows: In Section \ref{sec:2}, we recall some preliminaries, such as Gaussian Hilbert
space, Abstract Wiener spaces, reproducing covariance spaces, and a general greedy algorithm. Section \ref{sec:results} covers the main results. In Section \ref{sec:examples}, we present several basic examples, including numerical experiments, to demonstrate how the main theory can be applied. The proofs can be found in Section \ref{sec:proofs} and more general examples for noisy observations can be found in Section \ref{sec:6}.
\section{Preliminaries}\label{sec:2}

Throughout this work $(\Omega, \mathcal{A}, \mu)$ is a probability space and $J$ is an index set with at most countably many elements. We set $L^2(\mathcal{A}):=L^2(\Omega,\mathcal{A},\mu)$ as the Hilbert space of $\mu$-equivalence classes of square integrable functions. Note that here we emphasize the dependence on the $\sigma$-algebra $\mathcal{A}$ by taking it as argument of the space. The reason for this is that conditioning changes the $\sigma$-algebra, while the measure $\mu$ usually remains unchanged in this work. 

Moreover throughout this work $E$ and $F$ denote Banach spaces and $E'$  denotes the dual. We further write $B_E$ for the closed unit ball. Finally for all $A \subseteq E$ we write $\overline{A}$ for the closure of $A$.

We call a mapping $X:\Omega \to E$ weakly measurable, if $e'(X)$ is measurable for all $e' \in E'$. In addition we call $X$ strongly measurable if $X$ is measurable with respect to the Borel-$\sigma$ algebra generated by the norm topology on $E$ and $X(\Omega)$ is norm-separable.
In the case of separable Banach spaces both definitions coincide, as ensured by the well known Pettis-measurability theorem \cite[Theorem E.9]{cohn2013measure}, moreover we denote by $\mathcal{B}$ the Borel-$\sigma$ algebra of $E$. We call $X$ a random variable if $X$ is strongly-measurable.

Moreover recall that, given $p \geq 1$ and a Banach space $E$, the Bochner space $L^p(\mathcal{A},E)$ is the linear space of all $\mu$-equivalence classes of strongly measurable functions $f: \Omega \to E$ such that
\begin{align*}
\Vert f \Vert_{L^p(\mathcal{A},E) }^p :=\int_\Omega \Vert f \Vert_E^p \, \textup{d} \mu< \infty.
\end{align*}
For more details about Bochner spaces we refer to \cite[Chapter 1]{martingal}.

We continue with conditioning. 
To this end, let $\mathcal{A}_0$ be a sub-$\sigma$-algebra of $\mathcal{A}$ and $X \in L^1(\mathcal{A},E)$. 
An $(\mathcal{A}_0,\mathcal{B})$-measurable function $Z$ is called a conditional expectation of $X$  under $\mathcal{A}_0$ if  
    \begin{align*}
        \int_A Z \, \textup{d} \mu = \int_A X \, \textup{d} \mu, \quad \textup{for all} \quad A \in \mathcal{A}_0.
    \end{align*}
In this case we write $Z \in \mathbb{E}(X|\mathcal{A}_0)$.    
The existence and almost sure uniqueness of conditional expectations is guaranteed by \cite[Chapter 2.6]{martingal}.
For this reason we often do not distinguish between $Z$ and $\mathbb{E}(X|\mathcal{A}_0)$.
Given a random variable $Y:\Omega \to F$ we write $\sigma(Y)$ for the smallest $\sigma$-algebra such that $Y$ is weakly measurable. Moreover we set $\mathbb{E}(X|Y):=\mathbb{E}(X | \sigma(Y))$ and for $\mathfrak{F}'\subset F'$ we define 
 \begin{align*}
 \mathbb{E}(X| \mathfrak{F}' \circ Y ) := \mathbb{E}(X| \sigma ( \{f'(Y) \, : \, f' \in \mathfrak{F}' \} ) ).
 \end{align*}

We call a random variable $X:\Omega \to \mathbb{R} $ a one dimensional Gaussian random variable if there exists $m,\sigma \in \mathbb{R}$, such that $\mathbb{E}\left(\mathrm{e}^{\mathrm{i}tX} \right) = \mathrm{e}^{\mathrm{i}t m-\frac{1}{2}\sigma^2t^2}$     holds true for all $t \in \mathbb{R}$.
We then write $X \sim \mathcal{N}(m,\sigma^2)$.
We note that if $\sigma >0$ one obtains a normal distribution and if $\sigma =0$ one obtains a point probability measure.
We call a random variable $X: \Omega \to E$ Gaussian if $E$ is separable and $e'(X)$ is a real-valued one dimensional Gaussian random variable for all $e' \in E'$.
Given two random variables $X:\Omega \to E$ and $Y:\Omega \to F$, we call them jointly Gaussian if $(X,Y):\Omega \to E \times F; \,  \omega \mapsto (X(\omega),Y(\omega))$ is a Gaussian random variable.
If $X:\Omega \to E$ is a Gaussian random variable we have $X \in L^p(\mathcal{A},E)$ for all $p\geq 1$ by Fernique's Theorem \cite[Theorem 5.3]{vakhania2012probability}.
 	
We will often make the same assumptions, thus we consolidate them into a single statement:
\begin{assumptionA}\label{Assumption:A}
We require $E,F$ to be separable Banach spaces, $(\Omega,\mathcal{A},\mu)$ to be a probability space, and the random variables  $X:\Omega \to E$, $Y:\Omega \to F$ to be jointly Gaussian. Also we assume $\mathbb{E}(X)=\mathbb{E}(Y)=0$ and we set $Z:=\mathbb{E}(X|Y)$.
\end{assumptionA}
We remark that under \hyperref[Assumption:A]{Assumption A} $Z$ is indeed a Gaussian random variable with $\mathbb{E}(Z)=0$, see \cite[Theorem 3.3 vi)]{steinwart2024conditioning}.

We now introduce covariance operators, which are an essential tool for analyzing Gaussian random variables, see \cite[Chapter 1]{jansonTensor}.
    To this end, recall that under \hyperref[Assumption:A]{Assumption A} the \textup{covariance operator} $\textup{cov}(X):E' \to E$ of a Gaussian random variable $X$ is defined by $\textup{cov}(X)e':= \mathbb{E}(e'(X)X)$ for all $e' \in E'$.
Given $e' \in E'$ the one-dimensional conditional variance is defined by
 \begin{align*}
     \textup{cov}(e'(X)|Y) :=  \mathbb{E} \bigl((e'(X)-\mathbb{E}(e'(X)|Y))^2 | Y\bigr),
 \end{align*}
as in \cite[Chapter 4]{spanos2019probability}.
We want to generalize this to the infinite dimensional case. To this end recall that given $e_1,e_2 \in E$ the elementary tensor $e_1 \otimes e_2: E' \to E$ is given by $e' \to e'(e_1) e_2$. 
Moreover, the injective tensor product $E \check{\otimes}E$ is one instance of a Banach space containing all elementary tensors, see e.g. \cite[Section 4]{jansonTensor}. 
Now, given a random variable $X:\Omega \to E$ the pointwise defined map $X \otimes X: \Omega \to E \check{\otimes} E$ is also a random variable, see \cite[Theorem 6.7]{jansonTensor}.
\begin{definition}
    Under \hyperref[Assumption:A]{Assumption A} we define the \textup{conditional variance} $\textup{cov}(X|Y): E' \to E$ as 
    \begin{align*}
         \textup{cov}(X|Y)  := \mathbb{E}  \Bigl (  \bigl ( X-\mathbb{E}(X|Y)) \otimes (X-\mathbb{E}(X|Y) \bigr) | Y \Bigr ).
    \end{align*}
\end{definition}

\subsection{Gaussian Hilbert Spaces}
We give a short introduction into Gaussian Hilbert spaces. For more details, see \cite[Chapters 1 and 9]{Janson}.
A Gaussian linear space $G$ is a linear subspace of $L^2(\mathcal{A})$ such that each $g \in G$ is a Gaussian random variable with $\mathbb{E}(g)=0$.  A Gaussian Hilbert space $G$ is a Gaussian linear space that is closed in $L^2(\mathcal{A})$.
Under  \hyperref[Assumption:A]{Assumption A} we define $G_X:= \overline{\{e'(X) \, | \, e' \in E'\}}$
 as the Gaussian Hilbert space associated with the Gaussian random variable $X$.  
Under \hyperref[Assumption:A]{Assumption A}, we note that $(X,Y)$ is an $E \times F$-valued Gaussian random variable, and the space 
 \begin{align}\label{eq:G_{(X,Y)}}
     G_{(X,Y)}=  \overline{\{e'(X) + f'(Y) \, | \, (e',f') \in E' \times F' \}}    
 \end{align}
 is also a Gaussian Hilbert space. Note that by $\overline{\textup{span} \{ \cdots \}}$ we mean the closure in the $L^2(\mathcal{A})$-norm. 
More generally, for a set $A$ and a norm $\Vert \mycdot \Vert$, we write $\overline{A}^{\Vert \mycdot \Vert}$ for the closure of $A$ with respect to the norm $\Vert \mycdot \Vert$.

Conditioning in Gaussian Hilbert spaces $G$ reduces to orthogonal projections in $L^2(\mathcal{A})$. Specifically, for a set $ \mathcal{G}\subset G$  and $g \in G$ the conditional expectation of $g$ given $\mathcal{G}$ is
    \begin{align}\label{eq:Conditioning Orthogonal projection}
        \mathbb{E}(g \, | \mathcal{G}) :=  \mathbb{E}(g \, | \sigma( \mathcal{G})) = \Pi_{V} g,
    \end{align}
where $\Pi_{V}$ denotes the orthogonal projection onto the subspace $V:= \overline{\textup{span} (\mathcal{G})}$, see e.g. \cite[Theorem 3.3 vii)]{steinwart2024conditioning}.
Setting $g:=e'(X)$ we obtain with Theorem \ref{thm:BedingteErwartungenHüte} the equality $\mathbb{E}( e'(X) | Y) = \mathbb{E}(e'(X)|G_Y) = \Pi_{G_Y}(e'(X))$,
    with $\Pi_{G_Y}$ being the orthogonal projection from $G_{(X,Y)}$ onto the Gaussian Hilbert space $G_Y\subseteq G_{(X,Y)}$ and $\sigma(Y)=\sigma(G_Y)$ by Lemma \ref{lem:algebren gleich}.

\subsection{Abstract Wiener Spaces}
 We set the abstract Wiener space associated to the Gaussian random variable $X$ as
    \begin{align}\label{definiton}
        W_X:=\left\{ e \in E \, \middle| \, \exists g \in G_X \, \textup{with} \,  e= \int_\Omega g X \, \textup{d} \mu \right\}
    \end{align}
    with the norm being given by
    \begin{align}\label{Norm}
        \Vert e \Vert_{W_X} = \inf \left\{ \Vert g \Vert_{G_X } \, \middle| \, g \in G_X \, \textup{with} \,  e= \int_\Omega g X \textup{d} \mu  \right\}.
    \end{align}
In the literature abstract Wiener spaces are usually used for Gaussian measures, see e.g. \cite[Chapter 8]{stroock2010probability}.  
In any case $W_X$ is a Hilbert space, see Lemma \ref{lem:Unitar}. One can naturally switch between the spaces $G_X$ and $W_X$ by using a canonical isometric isomorphism $V_X: G_X \to W_X$ which is given by $V_Xg:= \int_\Omega g X \, \textup{d} \mu$ with $g \in G_X$, see Lemma \ref{lem:Unitar}.
    In particular the adjoint $V_X^*: W_X \to G_X$ of $V_X$ is given via $V_X^*=V_X^{-1}$, see Lemma \ref{lem:Unitar}. Furthermore if we extend $V_X$ to the whole space $L^2(\mathcal{A})$ by defining    $\hat{V}_X: L^2(\mathcal{A}) \to W_X$ via
$\hat{V}_X     g := \int_\Omega g X \, \textup{d} \mu$ for $g \in L^2(\mathcal{A})$,
then Lemma \ref{lem:Unitar} shows $\textup{ker}(\hat{V}_X) = G_X^\perp$ and that the adjoint $\hat{V}_X^*: W_X \to L^2( \mathcal{A})$ is given by $\hat{V}_X^*  = V_X^{-1}$, as in Lemma \ref{lem:Unitar}.
One can utilize $\hat{V}_X^*$ to map $W_X$ into $L^2(\mathcal{A})$ and then apply $\hat{V}_X$ to map it into $W_Y$. This establishes a natural connection between $W_X$ and $W_Y$, leading to the definition of the operator $L_W:W_X \to W_Y$ as follows:
\begin{equation}\label{eq:Definition of L_W as commutative diagramm}
\begin{tikzcd}
W_X \arrow[rr, "L_W"] \arrow[rd, "{\hat{V}_X^*}"] &                                                    & W_Y \\
                                                          & L^2(\mathcal{A}) \arrow[ru, "{\hat{V}_Y}"] &    
\end{tikzcd}
\end{equation}
We note that if $Y=LX+N$ where $L:E \to F$ is a bounded operator and $N:\Omega \to F$ is a Gaussian random variable independent of $X$, then $L_W=L|_{W_X}$. 
For details, see Lemma \ref{lem:L_W = L restricted}. 
Furthermore, it can be proven that for all $w_y \in W_Y$, we have $L_W^* w_y \in W_Z$, see Lemma \ref{lem:Conditioning X=Z}.
 Note that Lemma \ref{lem:Conditioning X=Z} shows $\textup{ran}(L_W^*) \subseteq W_Z$ and thus we can consider the new operator $M_W:W_Y \to W_Z$ given by 
 \begin{align}\label{eq:M W Definition}
 M_W w_y  :=L_W^* w_y = \hat{V}_Z \hat{V}_Y^* w_y.
\end{align}  
Note that $M_W$ is a bounded and linear operator with $\snorm{M_W}\leq 1 $.
In addition $M_W$ is surjective since 
$\hat{V}_Z:L^2(\mathcal{A}) \to W_Z$ is surjective and $\textup{ker}(\hat{V}_Z)^\perp=G_Z\subseteq G_Y=\textup{Im}(\hat{V}^*_Y)$, see Lemmata \ref{lem:Unitar} and \ref{lem:Subsets}.
Finally, $M_W$ is in general \emph{not} an isometry, see Lemma \ref{lem:isometry+surjective}.

\subsection{Reproducing Covariance Space}
We finally introduce a third Hilbert space $H_X$ associated to a Gaussian random variable $X$. To this end let $\iota: E \to E''$ be the canonical isometric embedding, that is $\iota(e)(e')=e'(e)$. 
Since $W_X \subseteq E$ we can thus define $H_X(E'):=\iota(W_X)$ and equip $H_X(E')$ with the scalar product of $W_X$ that is $\langle \iota f, \iota g \rangle_{H_X(E')}:=  \langle f,g \rangle_{W_X}$ with $f,g \in W_X$.
The space $(H_X(E'),\langle \mycdot , \mycdot \rangle_{H_X(E')})$ is an RKHS whose kernel is given by $k_{X,E'} (e'_1,e'_2) = \langle \textup{cov}(X) e'_1,e'_2 \rangle_{E,E'} $, see \cite{bach2023information}. Restricting the kernel $k_{X,E'}$ to the unit ball $B_{E'}$ gives the kernel $k_X: B_{E'} \times B_{E'} \to \mathbb{R}$, whose RKHS is $H_X:=    H_X(B_{E'}):=H_X(E')|_{B_{E'}}$, with the kernel then naturally given by $k_{X} (e'_1,e'_2) = \langle \textup{cov}(X) e'_1,e'_2 \rangle_{E,E'}$
for $e_1',e_2' \in B_{E'}$, see e.g. \cite[Lemma 4.3]{steinwart2008svm}.
Note that this introduces an isometric isomorphism $\iota_X: W_X \to H_X$ given by $\iota_X (w) = (\iota w) |_{B_{E'}},$
as well as a canonical isometric isomorphism $U_X:H_X \to G_X$ 
\begin{equation} \label{eq: U_X}
 U_X h := V_X^* \iota_X^* h=V_X^{-1} \iota_X^{-1} h.
\end{equation}
We can canonically and isometrically switch between all three spaces $G_X,W_X$ and $H_X$, meaning that e.g. for an orthogonal projection in $G_X$, we find corresponding orthogonal projections in $W_X$ and $H_X$.

To summarize these different spaces and operators we provide the following commutative diagram
\begin{equation*}
\begin{tikzcd}
 & W_X \arrow[rd, "{\iota_X}"] & \\
G_X \arrow[ru, "{V_X}"] & & H_X \arrow[ll, "{U_X}"]
\end{tikzcd}
\end{equation*}
and for the different purposes we provide the following table

\begin{table}[h]
\centering

\begin{tabular}{c c c}
\hline
{Hilbert space} &
{Contained in} &
{Purpose} \\
\hline
{$G_{(X,Y)}$} &
{$L^2(\mathcal{A})$} &
{Framework for studying conditioning.} \\

{$W_{(X,Y)}$} &
{$E \times F$} &
{Representation of the Gaussian random variables.} \\

{$H_{(X,Y)}$} &
{$E'' \times F''$} &
{Enables kernel methods.} \\
\hline
\end{tabular}

\caption{{Purpose of the different spaces associated with $(X,Y)$.}}
\end{table}

\subsection{$P$-Greedy}
Now we introduce the \textit{weak greedy algorithm} and recall the results from \cite{devore2013greedy} that we need. 
Let $H$ be a Hilbert space and let $\mathcal{F} \subseteq B_H$ be a compact subset. Let $\gamma \in ]0,1]$ be fixed. We first choose an element $f_0 \in \mathcal{F}$ such that $
 \gamma \max_{f \in \mathcal{F}} \Vert f \Vert_H \leq \Vert f_0 \Vert_H$ holds true.
Assuming $\{f_0, \cdots, f_{n-1} \}$ and $V_n:=\textup{span}\{f_0, \cdots, f_{n-1} \}$ have been selected, we then take $f_n \in \mathcal{F}$ such that 
\begin{align*}
\gamma \max_{f \in \mathcal{F}} \textup{dist}(f,V_n) \leq  \textup{dist}(f_n,V_n)_H.
\end{align*}

Now, we turn to the $P$-greedy algorithm. Consider an RKHS $H$ with kernel $k$ and domain $T$. In addition, let $\mathcal{F} := \{ k( \mycdot , t ) \, | \, t \in T \}$
be compact and $k(t,t) \leq 1$ for all $ t\in T$. 
We select points $(t_0,\cdots , t_{n-1}) \subseteq T$ via an iterative method. First we take $t_0\in T$ according to the criterion 
\begin{align*}
\gamma \sup_{t \in T}  \Vert k(\mycdot,t) \Vert_H \leq \Vert k(\mycdot,t_0) \Vert_H 
\end{align*}
with $\gamma \in ]0,1]$ and we set $T_1{:=} \{t_0 \}$. The next points are then selected according to
\begin{align*}
\gamma \sup_{t \in T}  \Vert k(\mycdot,t)- \Pi_{T_{n-1}} k(\mycdot,t)  \Vert_H \leq \Vert k(\mycdot,t_{n-1}) -\Pi_{T_{n-1}} k(\mycdot,t_{n-1}) \Vert_H,
\end{align*}
where $\Pi_{T_{n-1}}$ denotes the orthogonal projection onto the space  $\textup{span}\left(\{k(\mycdot,t) \, | \, t \in T_{n-1} \} \right)$ and again we set $T_n{:=}T_{n-1} \cup \{t_{n-1} \}$.
We call this point selection algorithm weak-$P$-greedy. The following Corollary is a consequence of the statements in \cite{santin2017convergence}.
\begin{corollary}\label{cor:convergence rate P-greedy}
Given an RKHS $H$ with kernel $k$ and a compact set $\mathcal{F} \subseteq H$ selecting points via the weak $P$-greedy one obtains 
\begin{align*}
\sup_{f \in B_H} \Vert f- \Pi_{T_n} f\Vert_{C(T)} \leq \sqrt{2}\gamma^{-1} \min_{1 \leq m <n } d_m^{\frac{n-m}{n}}(\mathcal{F}).
\end{align*}
with $d_m(\mathcal{F}) := \min_{V} \sup_{f \in \mathcal{F}} \Vert f - \Pi_{V} f \Vert_H$ denoting the Kolmogorov width of the set $\mathcal{F}$,
where the minimum is taken over all $m$-dimensional subspaces $V\subseteq H$ and $\Pi_{V}$ denotes the orthogonal projection in $H$ onto $V$.
\end{corollary}

We emphasize that the weak-greedy algorithm above does not determine a unique sequence of spaces as many elements might satisfy the weak greedy selection criterion. The inequality holds for any sequence obtained from the weak-$P$-greedy algorithm.

\section{Main Results}\label{sec:results}
In this section we present the main results of this work that investigate conditioning of Gaussian random variables. We start with a technical yet crucial theorem that allows us to essentially view the process of conditioning as orthogonal projections in the spaces $G_{(X,Y)},W_{(X,Y)}$ and $H_{(X,Y)}$.
\begin{theorem}\label{Theorem:Conditioning orthogonal projection} Under \hyperref[Assumption:A]{Assumption A} and given an ONB $(r_j )_{j \in J }$ of $G_Y$ the equality
    \begin{align}\label{eq:seriesrepresentation}
        \mathbb{E} (X| Y) = \sum_{j \in J } r_j \int_\Omega r_j X \, \textup{d} \mu 
    \end{align}
holds true, where the convergence of the series is almost everywhere and in $L^p(\mathcal{A},E)$ for all $p \in [1, \infty)$. Additionally the following statements hold true: 
\begin{enumerate}
\item The orthogonal projection $\Pi_G: G_{(X,Y)} \to G_{(X,Y)} $ onto $G_{(Z,Y)} =G_Y$ is given by
\begin{align*}
\Pi_G g = \mathbb{E} (g | Y ).
\end{align*}
\item The orthogonal projection $\Pi_W: W_{(X,Y)} \to W_{(X,Y)} $ onto $W_{(Z,Y)} $ is given by 
\begin{align*}
\Pi_W (w_x,w_y) = (L_W^* w_y,w_y). 
\end{align*} 
\item The orthogonal projection $\Pi_H: H_{(X,Y)} \to H_{(X,Y)} $ onto $H_{(Z,Y)} $ is given by
\begin{align*}
\Pi_H ( h_x,h_y) = (\iota_X L_W^* \iota^{-1}_Y h_y, h_y).
\end{align*}
\end{enumerate}
Finally, the orthogonal projections are related via the following relation
\begin{align*}
\Pi_G=V_{(X,Y)}^* \Pi_W V_{(X,Y)} = U_{(X,Y)} \Pi_H U_{(X,Y)}^*.
\end{align*}
\end{theorem}

Let us briefly relate Theorem \ref{Theorem:Conditioning orthogonal projection} to the existing literature. The work \cite[Theorem 3.5]{travelletti2024disintegration} considers Gaussian measures and provides a series representation similar to \eqref{eq:seriesrepresentation}. However, that result is formulated for noiseless observational data and the corresponding series is shown to converge almost surely. Moreover, \cite[Theorem 3.3]{owhadi2018conditioning} studies conditioning of Gaussian measures on a Hilbert space $H$ and shows, under suitable compatibility conditions on the Gaussian measure and the conditioning subspace, that the conditional expectation can be expressed by means of an orthogonal projection in $H$. In addition, \cite[Theorem 3.3]{steinwart2024conditioning} investigates conditioning of Gaussian random variables with an emphasis on regular conditional probabilities. There, the conditional expectation is obtained as a limit, and the connection between conditional expectation and orthogonal projection in Gaussian Hilbert spaces is also discussed.

Theorem \ref{Theorem:Conditioning orthogonal projection} establishes that conditioning corresponds to orthogonal projections in the spaces $G_{(X,Y)}$, $W_{(X,Y)}$, and $H_{(X,Y)}$. Moreover, it demonstrates that the operator $L_W^*$, and consequently $M_W$, are intrinsically linked to the conditioning process.


The next theorem uses the orthogonal projection statement of Theorem \ref{Theorem:Conditioning orthogonal projection} by showing that the kernel of the conditioned random variable is given by a projection of the initial kernel.


\begin{theorem}\label{Theorem: Kernel Conditioning}
    Under \hyperref[Assumption:A]{Assumption A} and given $\mathcal{E} \subset B_{E'}$, we set 
    \begin{align*}
        H(\mathcal{E}):= \overline{\textup{span}\{ k_{X}(\mycdot,e') \, | \, e' \in \mathcal{E}  \}}^{\Vert \mycdot \Vert_{H_{X}}}.
    \end{align*}
        Then $X_\mathcal{E}:= \mathbb{E} (X| \, \sigma \{ e'(X) \, : \, e' \in \mathcal{E} \} )$ is a Gaussian random variable, whose kernel is given by
   \begin{align*}
   k_{\mathcal{E}}(\mycdot ,e') = \Pi_{H(\mathcal{E})} k_X(\mycdot ,e'), \quad \textup{for all} \, \, e' \in B_{E'}.
   \end{align*}
Also, the kernel $k_{X-X_\mathcal{E}}: B_{E'} \times B_{E'} \to \mathbb {R}$  of the Gaussian random variable $X-X_\mathcal{E}$ is given by 
\begin{align*}
   k_{X-X_\mathcal{E}}(e'_1,e_2') = k_X(e'_1,e_2')-k_{\mathcal{E}}(e_1',e_2'), \quad \textup{for all} \, \, e_1' ,e_2' \in B_{E'}.
\end{align*}
Additionally the following equality holds true
    \begin{align*}
      \sup_{e' \in B_{E'}}  \Vert k_{X} (\mycdot,e') -  k_{\mathcal{E}}(\mycdot,e') \Vert_{H_X}^2 = \Vert \textup{cov}(X-X_\mathcal{E} )\Vert_{E'\to E}=\Vert \textup{cov}(X)-\textup{cov}(X_\mathcal{E} )\Vert_{E'\to E}. 
    \end{align*} 
\end{theorem}

For a short example of Theorem \ref{Theorem: Kernel Conditioning}, let $S \subseteq T \subseteq \mathbb{R}^d$ be compact,  $X: \Omega \to C(T)$ be a Gaussian random variable, and $\mathcal{E}:=\{ \delta_s \, | \, s \in S\}$ with $\delta_s$ being the point evaluation at $s$. Then to obtain $X_\mathcal{E}$ we only have to calculate $\Pi_{H(\mathcal{E})} k_X(\mycdot,\delta_t)$ for $t \in T$. Here, we note that $k_X(\delta_s,\delta_t)=k(s,t)$ holds true, where $k$ is the covariance function. We write $H$ for its RKHS. This gives $H=W_X$ and thus $\iota_X H=H_X$. Lastly we have 
\begin{align*}
H(\mathcal{E})=\iota H(S):= \iota_X \overline{\textup{span} \{ k(\mycdot, s) \, | \, s \in S\}}^{\Vert \mycdot \Vert_H}.
\end{align*}
In the setting of the short example, Theorem \ref{Theorem: Kernel Conditioning} can also be interpreted from a purely kernel-theoretic perspective, see \cite[Chapter 10.7]{wendland2004scattered}. There, it is shown that restrictions of RKHS functions admit natural minimal-norm extensions. 
 
\begin{corollary}\label{Coroallry:No Noise}
Under \hyperref[Assumption:A]{Assumption A} and given a bounded operator $L:E \to F$ such that 
\begin{align*}
Y=LX,
\end{align*}
one obtains $G_Z=G_Y \subseteq G_X$
and the reproducing kernel of $H_Z$ is given by 
\begin{align*}
k_Z(\mycdot,e')= \Pi_{H(\mathcal{E})} k_X(\mycdot,e')
\end{align*} 
for all $e' \in B_{E'}$ and $H(\mathcal{E}):= \overline{\textup{span}\{ k_{X}(\mycdot,L'f') \, | \, f' \in F' \wedge L'f' \in B_{E'}  \}}^{\Vert \mycdot \Vert_{H_{X}}}$.
\end{corollary}

Our next goal is to calculate the conditional expectation $Z=\mathbb{E}(X|Y)$. To this end we recall the mapping $M_W: W_Y \to W_Z$
\begin{align*}
        M_W w_y = \hat{V}_{Z} \hat{V}_{Y}^* w_y
\end{align*}
from \eqref{eq:M W Definition}.
Moreover, recall from Theorem \ref{Theorem:Conditioning orthogonal projection} that given an ONB $(r_j)_{j \in J} \subseteq G_Y$ we find
\begin{align*}
    Y=  \sum_{j \in J} r_j \int_\Omega r_j Y \, \textup{d}\mu.
\end{align*}
 If we could apply $M_W$ to $Y$, we would obtain 
\begin{align}\label{eq:M_Wcheating}
M_W Y = \sum_{j\in J} r_j M_W \left( \int_\Omega r_j Y \, \textup{d} \mu \right) = \sum_{j \in J} r_j \int_\Omega r_j Z \, \textup{d} \mu = Z,  
\end{align} 
thus solving the conditioning problem. Unfortunately, $M_W$ is only defined on $W_Y$, therefore Equation \eqref{eq:M_Wcheating} is not valid and should only be seen as a heuristic. 

In the following we investigate whether we can turn this heuristic into a rigorous argument.
To this end we explore under which conditions we can continuously extend the operator $M_W$. We begin with the following lemma.

\begin{lemma}\label{Lemma:M_W Bounded}
Let \hyperref[Assumption:A]{Assumption A} be satisfied and $W_Y$ be dense in $F$. If there exists a $C>0$ with
\begin{align}\label{eq:M_W bounded}
\Vert M_W w \Vert_E \leq C \Vert w \Vert_F
\end{align} 
for all $w \in W_Y$, then there exists a bounded linear operator $M:F \to E$ such that $MY=Z$ almost everywhere and $M|_{W_Y}=M_W$.
\end{lemma} 

An example where the inequality in \eqref{eq:M_W bounded} holds true is given in partial differential equation with well-posedness inequalities, e.g. \cite[Chapter 2.5]{lions2012non}.

If $M_W$ does not satisfy \eqref{eq:M_W bounded}, we can change one of the spaces $E$ or $F$ or both to obtain that $M_W$ can be extended, as we exemplify by the following theorem.

\begin{theorem}\label{Theorem:Extension of M_W}
    Under \hyperref[Assumption:A]{Assumption A}, assuming $G_Y=G_Z$, and $W_Y$ is dense in $F$, there always exists a Banach space $\tilde{E}$ 
    such that $W_Z\subseteq \tilde{E}$ and there exists an extension $\hat{M}:F \to \Tilde{E}$ of $M_W$ which is an isometric isomorphism and satisfies $\hat{M} Y=Z$.
\end{theorem}

We note that, instead of changing the norm on $E$, one could change the norm on $F$ to obtain a continuous extension of $M_W$ into the space $E$.

Having shown criteria for extending the operator $M_W$ to an operator $M:F \to E$, we make the following assumption:
\begin{assumptionM}\label{Assumption:M}
    Under \hyperref[Assumption:A]{Assumption A}, we assume that there exists a bounded linear operator $M:F \to E$ such that 
$MY=Z$ almost everywhere holds true.
\end{assumptionM}
The operator $M$ allows us to obtain an easy formula for the conditional variance.
\begin{theorem}\label{Theorem:conditional_variance}
    Under \hyperref[Assumption:A]{Assumption A} the conditional variance of $X$ is given via 
    \begin{align*}
       \textup{cov}(X|Y)=  \textup{cov}(X) - \textup{cov}(Z) 
    \end{align*}
    with $Z:=\mathbb{E}(X|Y)$. Additionally under \hyperref[Assumption:M]{Assumption M} the conditional variance is also given via
    \begin{align*}
         \textup{cov}(X|Y) =  \textup{cov}(X) -  M\textup{cov}(Y) M'.
    \end{align*}
\end{theorem}
Theorem \ref{Theorem:conditional_variance} recovers the well-known formula for the conditional covariance of Gaussian random variables. In the finite-dimensional case, for instance when \(F=\mathbb{R}^n\) and \(\operatorname{cov}(Y)\) is invertible, one obtains
\begin{align*}
\operatorname{cov}(X\mid Y)
=
\operatorname{cov}(X)
-
\operatorname{cov}(X,Y)\operatorname{cov}(Y)^{-1}\operatorname{cov}(Y,X),
\end{align*}
see, \cite[Equation (3)]{wilson2021pathwise}. The formulation in Theorem \ref{Theorem:conditional_variance} avoids the explicit inversion of \(\operatorname{cov}(Y)\). Instead, the conditional covariance is expressed in terms of the operator \(M\), which represents the conditional mean. This is advantageous in infinite-dimensional settings, where \(\operatorname{cov}(Y)\) is typically not boundedly invertible, if an inverse exists on its range, it is generally unbounded.

We use Theorem \ref{Theorem:conditional_variance} to derive convergence rates for the conditional variance leading to the following theorem.

\begin{theorem}\label{Theorem:main convergence rate generalized}
    Under \hyperref[Assumption:A]{Assumption A}, \hyperref[Assumption:M]{Assumption M}, and given a subset $F_n' \subset F'$.  We define 
   \begin{align*}
    Y_n&{:=}\mathbb{E}(Y| \sigma \{ f'(Y) \, : \, f' \in F_n' \} ).   
   \end{align*}
We obtain 
    \begin{align*}
        \Vert \textup{cov}(X|Y)-\textup{cov}(X|Y_n) \Vert_{E' \to E} \leq  \Vert M \Vert_{F\to E}^2 \Vert \textup{cov}(Y) - \textup{cov}(Y_n) \Vert_{F' \to F}.
    \end{align*}
    \end{theorem}

In summary, Theorem \ref{Theorem:main convergence rate generalized} states that if $M:F \to E$ is bounded then the convergence rates of $\textup{cov}(Y)-\textup{cov}(Y_n)$ translate into convergence rates of $\textup{cov}(X|Y)-\textup{cov}(X|Y_n)$. The following corollary demonstrates how this can be done by applying the weak-$P$-greedy algorithm on $Y$.  
   
\begin{corollary}\label{Corollary:Polynomial Convergence Rate P-Greedy}
Assume that we are in the setting of Theorem \ref{Theorem:main convergence rate generalized} with \hyperref[Assumption:M]{Assumption M} and that the statement $\Vert\textup{cov}(Y)\Vert_{F'\to F} \leq 1$ holds true. 
If we obtain the set $F_n'$ by applying the weak-$P$-greedy method on $k_Y$, then we obtain the inequality
    \begin{align}\label{eq:mainresult1}
        \Vert   \textup{cov}(X|Y) - \textup{cov}(X|Y_n) \Vert_{E' \to E} \leq  2 \Vert M \Vert_{F\to E}^2 \min_{1\leq m <n} 
\gamma^{-2}    d_m^\frac{2(n-m)}{n} ({\mathcal{F}}),
    \end{align} 
with $\mathcal{F}=\left\lbrace k_Y(\mycdot,f') \, \middle| \, f' \in B_{F'} \right\rbrace$.
Lastly assuming there exist $C,\alpha>0$ such that for each $n \in \mathbb{N}$ there exists a set $\mathfrak{F}_n' \subset F'$ with $|\mathfrak{F}_n'|=n$ and $Y_n^\star:= \mathbb{E}(Y| \sigma \{ f'(Y) \, : \, f' \in \mathfrak{F}_n' \} )$
\begin{equation}\label{eq:polynomial convergence rate}
\Vert \textup{cov}(Y)-\textup{cov}(Y_n^\star) \Vert_{F'\to F} \leq C n^{-\alpha}
\end{equation}
then the inequality
\begin{align*}
 \Vert   \textup{cov}(X|Y) - \textup{cov}(X|Y_n) \Vert_{E' \to E} \leq   \Vert M \Vert_{F\to E}^2  
2^{5\alpha+1}\gamma^{-2} C n^{-\alpha}
\end{align*}
holds true.
\end{corollary}

Corollary \ref{Corollary:Polynomial Convergence Rate P-Greedy} shows that we obtain at least the optimal convergence rate of $\textup{cov}(Y)$ for $\textup{cov}(X|Y)$ by applying the weak-$P$-greedy algorithm on $Y$. 
A short example when the Inequality \eqref{eq:polynomial convergence rate} is satisfied, is given by $T \subseteq \mathbb{R}^d$ compact, $F=C(T)$ and $W_Y$ is equal to a Sobolev space on $T$ with regularity $s>d/2$. Then the sampling inequalities in \cite{rieger2008sampling} give \eqref{eq:polynomial convergence rate} with $\alpha=2s/d-1$.

\begin{remark}
If one replaces the Banach space $E$ by a Banach space $\tilde{E}$, as mentioned in Theorem \ref{Theorem:Extension of M_W} the convergence rate results from Theorem \ref{Theorem:main convergence rate generalized} still hold. However $\textup{cov}(X|Y)$ is not necessarily a mapping from $\tilde{E}'$ to $\tilde{E}$ but the difference $\textup{cov}(X|Y)-\textup{cov}(X|Y_n)$ is a mapping from $\tilde{E}'$ to $ \tilde{E}$. 
\end{remark} 

\section{Examples}\label{sec:examples}\label{sec:4}
In this section we give a few examples that show how to calculate $M_W$ and when $M_W$ can be extended and when not.
We begin with a positive example in which the spaces $W_X$ and $W_Y$ are well known and understood.

\begin{example}\label{example:4}
Let $E:=C([0,1])$, $X$ be the Brownian motion, and 
$L: C([0,1]) \to C([1/2,1])$ be the restriction operator on the interval $[1/2,1]$, that is
\begin{align*}
Lf := f |_{[1/2,1]}.
\end{align*} 
Let us further consider $Y:= LX$. Then $W_X$ is an RKHS with kernel 
\begin{align*}
k_{W_X}(t,s)=\min (s,t)
\end{align*}
and its scalar product is given by 
\begin{align*}
\langle u,v \rangle_{W_X} = \int_0^1 u' (t) v' (t) \, \textup{d} t,
\end{align*}
see e.g.~\cite[Subsection 8.1.2]{stroock2010probability}. 

We note that by $Y=LX$ we have that $W_Y=LW_X$ is an RKHS and its kernel is given by  
\begin{align*}
    k_{W_Y}(s,t) = \min(s,t)
\end{align*}
for all $s,t\in [a,b]$ with $a:= 1/2$ and $b:= 1$.
By Lemma \ref{lem:minimumkernel} we thus  have 
\begin{align*}
\langle u,v \rangle_{W_Y} =  2\cdot u(1/2)v(1/2) + \int_{1/2}^1 u' (t) v' (t) \, \textup{d} t
\end{align*} 
for all $u,v \in W_Y$.
Let us now calculate the adjoint of $L_W: W_X \to W_Y$, which is given  by
\begin{align*}
L_W w = w|_{[1/2,1]},
\end{align*} 
see Lemma \ref{lem:L_W = L restricted}.
To this end, we first note that for 
$u \in W_X$ and $v \in W_Y$
we have 
\begin{align*}
  2\cdot u(1/2)v(1/2) + \int_{1/2}^1 u' (t) v' (t) \, \textup{d} t 
  =
  \langle L_W u , v \rangle_{W_Y} 
  = 
  \langle u, L_W^* v \rangle_{W_X}
  =
  \int_0^1 u'(t) (L_W^*v)'(t) \, \textup{d} t.
\end{align*}
In particular, for $u\in W_X$ with $\textup{supp}(u) \subset [1/2,1]$ we can conclude that 
\begin{align*}
\int_{1/2}^1 u' (t) v' (t) \, \textup{d} t 
= 
\int_{1/2}^1 u' (t) (L_W^*v)' (t) \, \textup{d} t. 
\end{align*}
Recognizing that for all  $\tilde u \in L^2([1/2,1])$ we find an $u\in W_X$ with $\textup{supp}(u) \subset [1/2,1]$ and 
$u'=\tilde u$,
we conclude  $v'(s) =(L_W^*v)'(s)$ for almost all $s\in [1/2,1]$.  
Combining the calculations leads to
\begin{align}\label{eq:ScalarproductWithoutNoise}
 2\cdot u(1/2)v(1/2) = \int_0^{1/2} u'(t) (L_W^*v)'(t) \, \textup{d}t
\end{align} 
for all $u\in W_X$ and $v\in W_Y$.
Using the fundamental theorem of calculus and $u(0)=0$ we conclude
\begin{align*}
u(1/2)= \int_0^{1/2} u'(t) \, \textup{d} t .
\end{align*}
Combining this with \eqref{eq:ScalarproductWithoutNoise} leads to 
\begin{align*}
2 \int_0^{1/2} u'(t) v(1/2) \, \textup{d}t = \int_0^{1/2} u'(t) (L_W^*v)'(t) \, \textup{d}t.
\end{align*}
Again we can conclude that for almost all $s\in [0,1/2]$ we have $(L_W^*v)'(s)= 2\cdot v(1/2)$.
Putting these results together and taking into consideration that $L_W^*v$ has to be continuous with $L_W^*v(0)=0$, we end up with
\begin{align}\label{eq:LWAdjointBrownianMotion}
(L_W^*v)(s)= \begin{cases}
2 v(1/2) s, & \qquad \textup{for} \, \, s\leq 1/2, \\
v(s), & \qquad \textup{for} \, \, s>1/2.
\end{cases}
\end{align}
We note for $v\in W_Y$ we have $\inorm{L_W^*v}\leq \inorm v$, and therefore
 $L_W^*$ can be uniquely extended to a bounded linear operator $M: C([1/2,1]) \to C([0,1]) $.
 Finally, it is not hard to see that $Mv$ can be calculated as in \eqref{eq:LWAdjointBrownianMotion}.
See \hyperref[fig:1]{Figure: 1} and \hyperref[fig:2]{Figure: 2} for a numerical example by applying the $P$-greedy algorithm onto $Y$.

\begin{figure}[H]\label{fig:1}
  \centering
  \includegraphics[width=0.8\textwidth]{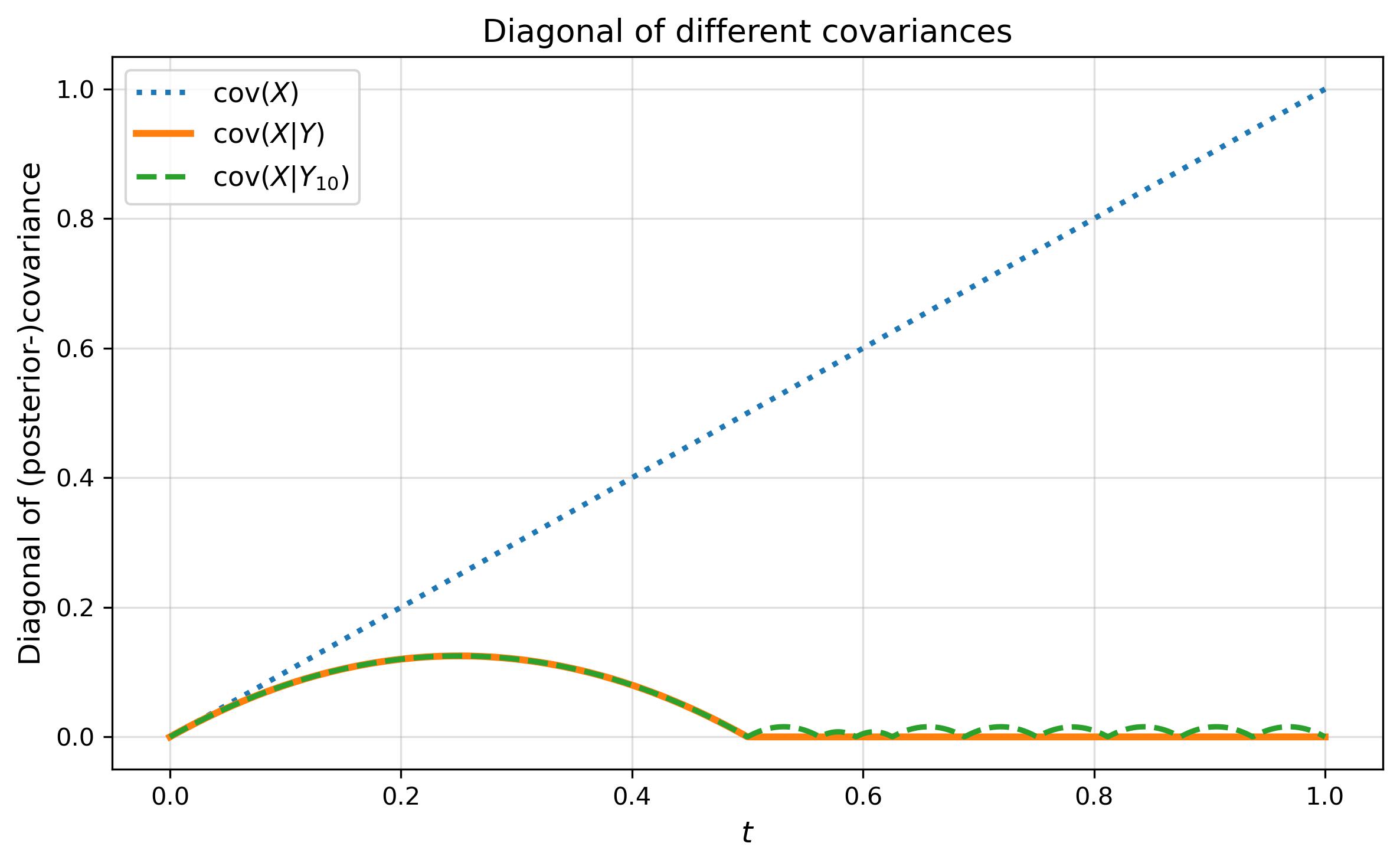}
  \caption{This plot shows the diagonal evaluated at $\delta_t$ with $t \in [0,1]$ of the different covariances meaning $\langle \textup{cov}(X|Y)\delta_t,\delta_t \rangle_{E,E'}$, $\langle \textup{cov}(X|Y_{10})\delta_t,\delta_t \rangle_{E,E'}$ and $\langle \textup{cov}(X)\delta_t,\delta_t \rangle_{E,E'}$ with $Y_{10}$ being $Y$ conditioned on $10$ points selected by the weak-$P$-greedy method with $\gamma=1$. With $X$ and $Y$ being as in Example \ref{example:4}.}
  
\vspace*{0.5cm}

  \centering
  \includegraphics[width=0.8\textwidth]{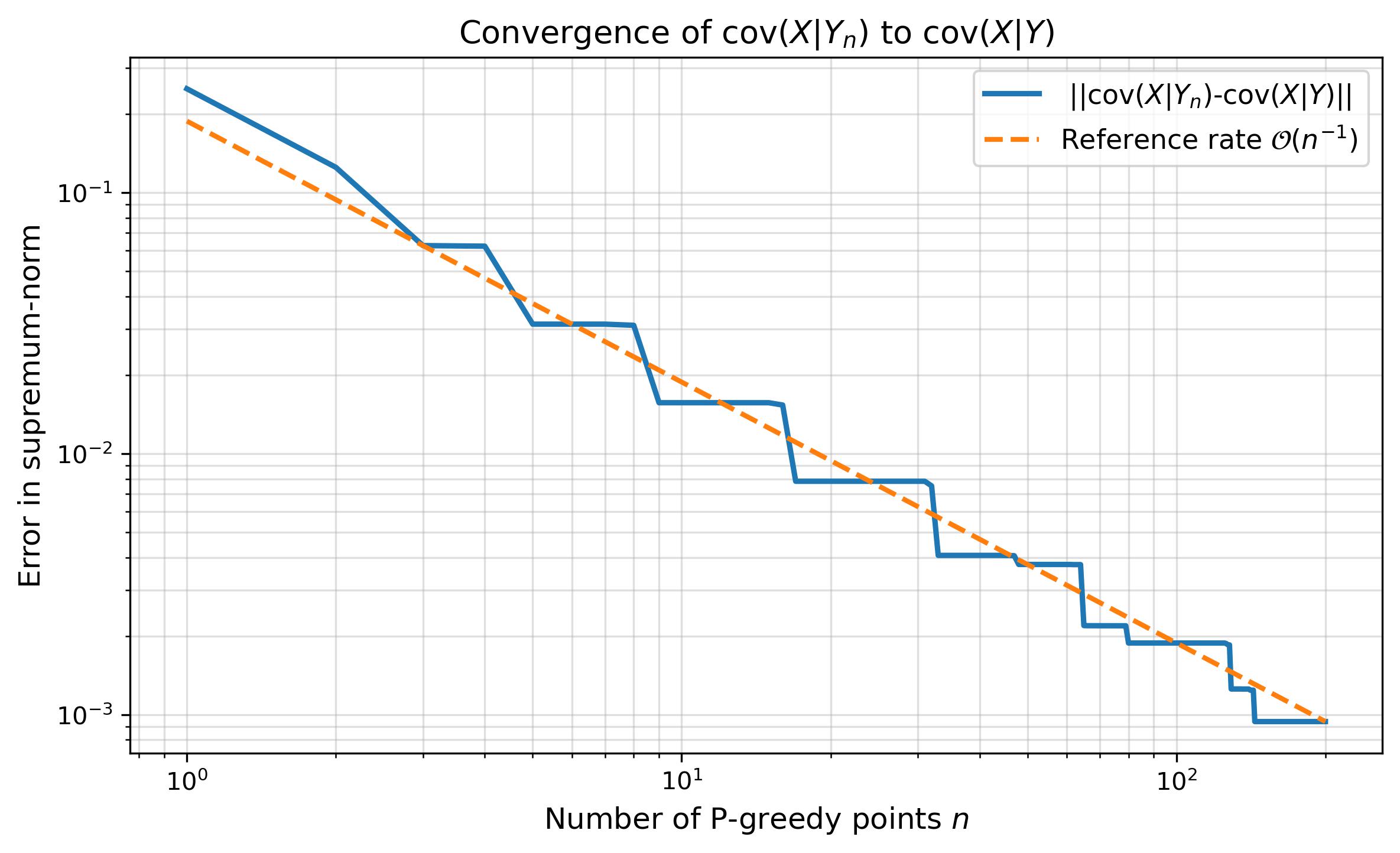}\label{fig:2}
  \caption{This plot shows $\Vert \textup{cov}(X|Y_n) - \textup{cov}(X|Y) \Vert_{E' \to E}$ with $Y_n$ being $Y$ conditioned on $n$ points selected by the $P$-greedy method. We observe a decay rate of $n^{-1}$, which matches the rate that is expected in the Sobolev setting, see \cite{santin2017convergence} combined with Corollary \ref{Corollary:Polynomial Convergence Rate P-Greedy}.
  The plot resembles a step function because, for Brownian motion, adding a single observation point reduces the error only locally, whereas the supremum norm measures the maximum error over the entire domain.}
\end{figure}
\end{example}

In the following example, we add some noise to the Gaussian random variable of 
Example \ref{example:4}.

\begin{example}\label{example:5}
Let $X$ and $L$ be as in 
 Example \ref{example:4} and  
\begin{align*}
Y:=LX+N,
\end{align*}
where $N$ is a Gaussian random variable independent of $X$, whose  kernel is given by
\begin{align*}
k_{W_N}(s,t)=\sigma^2.
\end{align*}
This changes the kernel of $W_Y$ to 
\begin{align*}
k_{W_{Y}}(s,t)=\sigma^2+ \min(s,t),
\end{align*}
and by Lemma \ref{lem:minimumkernel} the scalar product is therefore given by 
\begin{align*}
\langle u, v \rangle_{W_{Y}} = \frac{u(1/2)v(1/2)}{1/2+\sigma^2}+ \int_{1/2}^1 u'(t) v'(t) \, \textup{d} t.
\end{align*}
Repeating the steps of  Example \ref{example:4} up to \eqref{eq:ScalarproductWithoutNoise},  
we end up with $v'(s) =(L_W^*v)'(s)$ for almost all $s\in [1/2,1]$ as well as
\begin{align*}
 \frac{ u(1/2)v(1/2)}{1/2+\sigma^2} = \int_0^{1/2} u'(t) (L_W^*v)'(t) \, \textup{d}t.
\end{align*} 
Again we use the fundamental theorem of calculus to conclude  
\begin{align*}
\frac{1}{1/2+\sigma^2} \int_0^{1/2} u'(t) v(1/2) \, \textup{d}t = \int_0^{1/2} u'(t) (L_W^*v)'(t) \, \textup{d}t.
\end{align*}
Thus for almost all $s \in [0,1/2] $ we have 
\begin{align*}
(L_W^*v)'(s)=\frac{v(1/2)}{1/2+\sigma^2}.
\end{align*}
Combining these considerations and respecting the continuity of $L_W^*v$ and $(L_W^*v)(0)=0$ we end up with
\begin{align*}
(L_W^*v)(s)= \begin{cases}
\frac{ v(1/2)}{1/2+\sigma^2} \cdot s , & \qquad \textup{for} \, \, s\leq 1/2 \\
 v(s)- \frac{\sigma^2}{1/2+\sigma^2} v(1/2), & \qquad \textup{for} \, \, s>1/2.
\end{cases}
\end{align*}
We note that $L_W^*$ can also be extended onto the set of continuous functions as in Example \ref{example:4}. 
We refer to \hyperref[fig:3]{Figure: 3} and \hyperref[fig:4]{Figure: 4} to for a numerical example by applying the $P$-greedy algorithm onto $Y$ with $\sigma^2=1/5$.
\end{example}

\begin{figure}[H]
  \centering
  \includegraphics[width=0.8\textwidth]{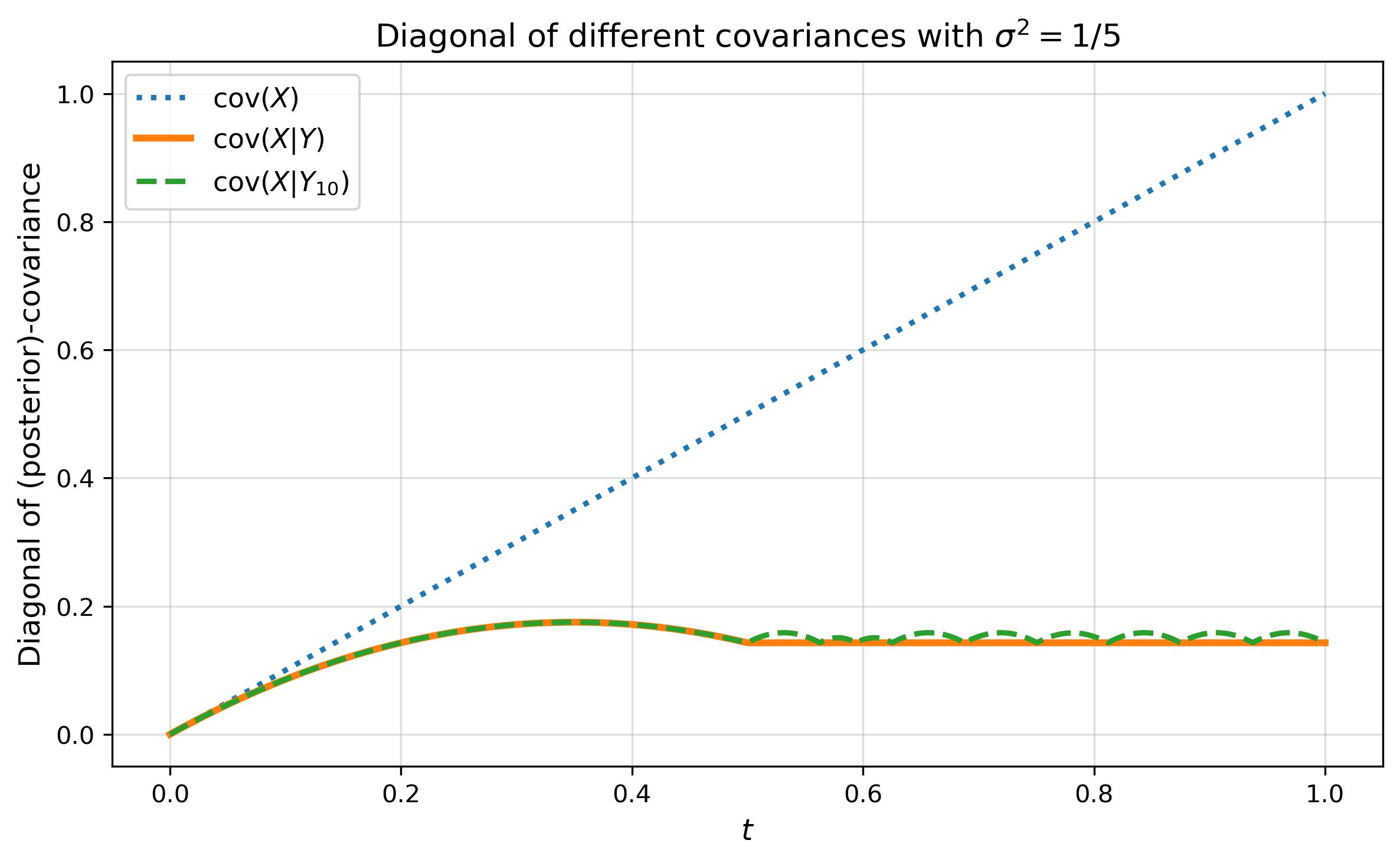}\label{fig:3}
  \caption{This plot shows the diagonal evaluated at $\delta_t$ with $t \in [0,1]$ and $\sigma^2=1/5$ of the different covariances meaning $\langle \textup{cov}(X|Y)\delta_t,\delta_t \rangle_{E,E'}$, $\langle \textup{cov}(X|Y_{10})\delta_t,\delta_t \rangle_{E,E'}$ and $\langle \textup{cov}(X)\delta_t,\delta_t \rangle_{E,E'}$ with $Y_{10}$ being $Y$ conditioned on $10$ points selected by the weak-$P$-greedy method with $\gamma=1$. Note that $\langle \textup{cov}(X|Y) \delta_t , \delta_t \rangle>0$ for $t>0$.}

\vspace{0.5cm}

  \centering
  \includegraphics[width=0.8\textwidth]{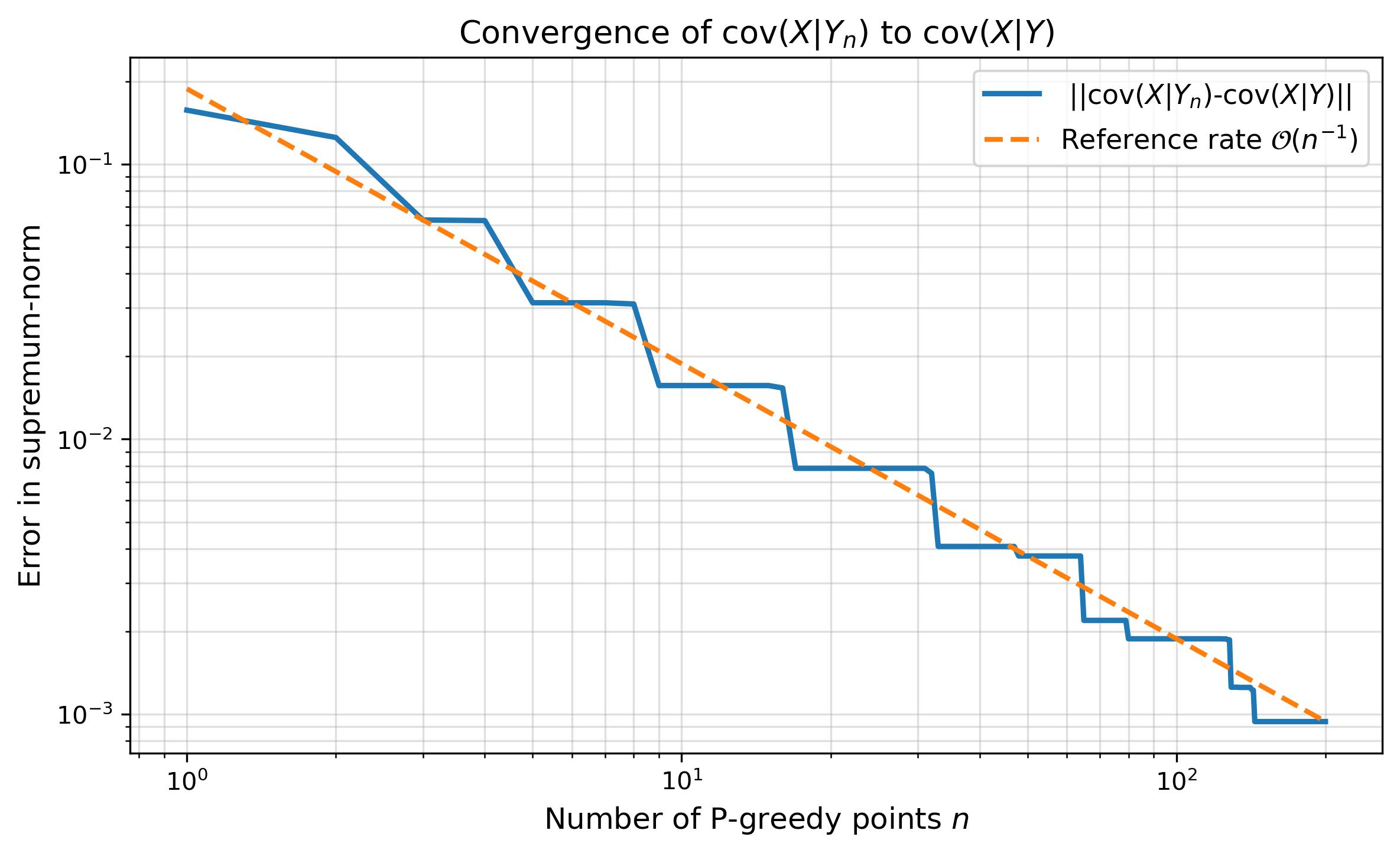}\label{fig:4}
  \caption{This plot shows $\Vert \textup{cov}(X|Y_n) - \textup{cov}(X|Y) \Vert_{E' \to E}$ with $\sigma^2=1/5$ and $Y_n$ being $Y$ conditioned on $n$ points selected by the $P$-greedy method. We observe a decay rate of $n^{-1}$, which matches the rate that is expected in the Sobolev setting, see \cite{santin2017convergence} combined with Corollary \ref{Corollary:Polynomial Convergence Rate P-Greedy}. We obtain the same rate as in the noiseless case. The step-like pattern can again be explained by the fact that adding a single observation point reduces the error only locally.}
\end{figure}

The last example provides a situation in which the operator $M_W$ \emph{cannot} be extended.

\begin{example}\label{example:Gegenbeispiel}
Let $E:=C^1([0,1])$ be the space of once continuously differentiable functions and $F:=C([0,1])$ be the space of continuous functions.
Moreover, let $X$ be an $E$-valued Gaussian random variable such that $W_X$ is dense in $E$ and 
\begin{align*}
Y:=\textup{Id} (X),
\end{align*}
where $\textup{Id}:E \to F$ is the embedding. 
Note that we have $\textup{Id}(W_X)=W_Y$, and therefore we can define $\textup{Id}_W:W_X \to W_Y$ as $\textup{Id}_W (w) = \textup{Id}(w)$ for all $w \in W_X$. For $w \in W_X$ and $g:=V_X^*w$ we thus find
\begin{align*}
L_W w = \hat{V}_Y \hat{V}_X^* w = \hat{V}_Y g = \int_\Omega g Y \, \textup{d} \mu = \int_\Omega g \textup{Id}(X) \, \textup{d} \mu = \textup{Id}_W \left( \int_\Omega g X \, \textup{d} \mu \right) = \textup{Id}_W (w).
\end{align*}
Moreover, setting $v := \textup{Id}_W (w) \in W_Y$ and $h:=V_Y^* v$, we find 
\begin{align*}
V_Y h = V_Y V_Y^* \textup{Id}_W(w) = \textup{Id}_W(w) =\hat{V}_Y g,
\end{align*}
where in the last step we re-used parts of the previous calculation. Applying $V_Y^*$ on both sides leads to $h=g$. 
Consequently we obtain
\begin{align*}
\Vert \textup{Id}_W(w) \Vert_{W_Y} = \Vert v \Vert_{W_Y} = \Vert h \Vert_{G_Y}= \Vert g \Vert_{G_X} = \Vert w \Vert_{W_X}.
\end{align*}
In other words, $\textup{Id}_W$ is an isometric isomorphism and thus $M_W=L_W^*=\textup{Id}_W^*=\textup{Id}_W^{-1}$. 
Now observe that $\textup{Id}_W^{-1}:W_Y \to W_X$ is given by $\textup{Id}_W^{-1} w = w$, where both sides are viewed as functions.
Moreover, since $W_X$ is dense in $E$ and $E$ is dense in $F$ we quickly see that $W_Y$ is dense in $F$. 
Consequently the mapping $\textup{Id}_W^{-1}$ is not continuous with respect to the norms $\Vert \mycdot \Vert_{E}$ and $\Vert \mycdot \Vert_F$ and therefore $M_W$ cannot be extended to a desired $M$.  
To see this more exactly. Let us assume that there exists a bounded extension $M:F \to E$ of $M_W$. Consider $\textup{Id} \circ M : F \to F$ by the denseness argument we have that $\textup{Id} \circ M=\textup{Id}_F$, with $\textup{Id}_F:F \to F$ being the identity of $F$. This implies that $\textup{Id}(E)=F$ which in our setting is a contradiction.

In view of Theorem \ref{Theorem:Extension of M_W} we note that for $\tilde{E}:=F$ we obtain $M w=w$ as a continuous linear mapping, see Example \ref{example:1}. 
\end{example}

\section{Proofs}\label{sec:proofs}\label{sec:5}

We will often utilize statements from \cite{steinwart2024conditioning}, for this we are going to summarize them here. 
First, we introduce the basic framework that allows us to understand the results of \cite{steinwart2024conditioning} that we will use. 

\begin{definition}
Let $F$ be a separable Banach space and $(A_n)$ be a sequence of bounded linear operators $A_n:F \to \mathbb{R}^d$. Then we say that $(A_n)$ is a filtering sequence for $F$, if we have $\sigma(A_n) \subseteq \sigma(A_{n+1})$ for all $n \geq 1$ and 
\begin{align*}
\sigma(F')=\sigma(A_n \, \vert \, n \geq 1).
\end{align*}
\end{definition}

The next theorem is part of the main result of \cite{steinwart2024conditioning} that we are going to use.  

\begin{theorem}\label{the:Theorem3.3Steinwart}
Let \hyperref[Assumption:A]{Assumption A} be satisfied. Moreover, let $(A_n)$ be a filtering sequence for $F$ and 
\begin{align*}
Y_n:=A_n \circ Y, \qquad n \geq 1.
\end{align*}
We then have almost surely
\begin{align*}
Z = \lim_{n \to \infty} \textup{cov}(X,Y_n) (\textup{cov}(Y_n))^\dagger Y_n,
\end{align*} 
where the limit is taken in $E$ and $(\textup{cov}(Y_n))^\dagger$ denotes the Moore-Penrose inverse of $\textup{cov}(Y_n)$, see \cite{ding1994perturbation}. For the covariance we have 
\begin{align*}
\textup{cov}(Z)=\textup{cov}(X)-\textup{cov}(X\vert Y).
\end{align*}  
Additionally we have $G_Z \subseteq G_Y$ and the orthogonal projection $G_{(X,Y)} \to G_{(X,Y)}$ onto the subspace $G_Y$ is given by 
\begin{align}\label{eq:ConditioninginG}
 \Pi_{G_Y} g = \mathbb{E}(g \vert Y), \qquad g \in G_{(X,Y)}.
\end{align}
\end{theorem}
For a proof we refer to \cite[Theorem 3.3 vi) + vii)]{steinwart2024conditioning}. We emphasize that \eqref{eq:ConditioninginG} characterizes the conditioning of Gaussian random variables within the Gaussian Hilbert space framework, which will be used extensively throughout this paper. The next results are auxiliary in nature and primarily address technical aspects. 

\begin{lemma}\label{lem:LemmaB2Steinwart}
Let $E$ be a separable Banach space. 
Given two Gaussian random variables $X,\tilde{X}:\Omega \to E$ such that
\begin{align*}
\langle e', \textup{cov}(X) e' \rangle_{E',E}= \langle e', \textup{cov}(\tilde{X}) e' \rangle_{E',E}, \qquad \forall e' \in E'.
\end{align*}
Then we have $\textup{cov}(X)=\textup{cov}(\tilde{X})$. 

Let \hyperref[Assumption:A]{Assumption A} be satisfied and given two separable Banach spaces $\tilde{E}, \tilde{F}$ and two bounded operators $\tilde{S}: E \to \tilde{E}$ and $\tilde{T}: F \to \tilde{F}$ we then have 
\begin{align*}
\textup{cov}(\tilde{S} \circ X, \tilde{T} \circ Y) = \tilde{S} \circ \textup{cov}(X,Y) \circ \tilde{T}' , 
\end{align*} 
with $\tilde{T}'$ denoting the adjoint of $\tilde{T}$.
\end{lemma}
For a proof we refer to \cite[Lemma B.2]{steinwart2024conditioning} and \cite[Equation B.3]{steinwart2024conditioning}. 

\begin{lemma}\label{lem:LemmaC3Steinwart}
Let $(X_n)$ be a sequence of $\mathbb{R}$-valued Gaussian random variables and $X$ be an $\mathbb{R}$-valued random variable with $X_n \to X$ in distribution. Then $X$ is a Gaussian random variable.
\end{lemma}
For a proof we refer to \cite[Lemma C.3]{steinwart2024conditioning}.

\begin{lemma}\label{lem:Unitar}
     Let \hyperref[Assumption:A]{Assumption A} be satisfied and $   V_X: G_X \to W_X $  be the map given by
         \begin{align*}
       V_X g := \int_\Omega g X \, \textup{d} \mu \, .
    \end{align*}
Then the following statements hold true:
\begin{enumerate}
    \item $V_X$ is an isometric isomorphism and its  adjoint $V_X^*$ of $V_X$ is given by $V_X^* = V_X^{-1}$.
    \item $W_X$ is a Hilbert space.
    \item The operator $V_X$ can be extended to an operator $ \hat{V}_X: L^2(\mathcal{A}) \to W_X $ given by 
 \begin{align*}
  \hat{V}_X   g := \int_\Omega g X \, \textup{d} \mu
 \end{align*} 
 and we have  $\textup{ker}(\hat{V}_X) = G_X^\perp$. Additionally its adjoint $\hat{V}_X^*$ of $\hat{V}_X$ is given by $\hat{V}_X^*=V_X^{-1}$.
\end{enumerate}
 \end{lemma} 
 
\begin{proof}
\ada i
     First we note that  $\text{Im} (V_X)=W_X$ directly follows from the definition of $W_X$, as stated in equation \eqref{definiton}. 
     Moreover, for the proof of the injectivity we  fix a $g \in \text{ker}(V_X)$. Then we have
     \begin{align*}
       0 =  V_X g = \int_\Omega g X \, \textup{d} \mu,
     \end{align*} 
     and consequently, for all $e' \in E'$, we obtain
     \begin{align*}
       0 =   e'(V_Xg) = \int_\Omega g e'(X) \, \textup{d} \mu.
     \end{align*} 
     This in turn implies $g \in G_X^\perp = \{0\}$, where $G_X^\perp$ is the orthogonal complement of $G_X$ in itself.   
  
    Since we have just seen that $V_X$ is bijective, we now conclude that it is isometric
    by the very definition \eqref{Norm} of the norm of $W_X$.

\ada {ii} Directly follows from \emph{i)}.


\ada {iii} We now write 
    $G_X^\perp$ for the orthogonal complement of $G_X$ in $L^2({\mathcal A})$. Then,  for all $g \in L^2({\mathcal A})$, there exist unique $g_X \in G_X$ and $g_X^\perp \in G_X^\perp$ such that 
  \begin{align}\label{eq:g=gX+gXperp}
      g=g_X + g_X^\perp \, .
  \end{align}
Note that for all $e' \in E'$ we have 
\begin{align*}
   e'\left( \int_\Omega  g_X^\perp X \, \textup{d} \mu \right) = \int_\Omega  g_X^\perp e'(X)  \textup{d} \mu = 0,
\end{align*}
and hence Hahn-Banach's theorem shows 
\begin{align*}
     \int_\Omega  g_X^\perp X \, \textup{d} \mu  = 0\, .
\end{align*}
Using $g=g_X + g_X^\perp$ we then see that $\hat V_X g \in W_X$ and $\textup{ker}(\hat{V}_X) = G_X^\perp$. 

To prove the last assertion, we fix a $w\in W_X$ and a $g\in L^2(\mathcal{A})$ with 
\eqref{eq:g=gX+gXperp}.
Then we have 
\begin{align*}
    \langle g, \hat V_X^* w\rangle_{L^2({\mathcal A})}
    = 
    \langle \hat V_Xg,   w\rangle_{W_X}
    = 
    \langle  V_Xg_X,   w\rangle_{W_X}
    = 
    \langle g, V_X^{-1} w\rangle_{L^2({\mathcal A})}\, ,
\end{align*}
where in the second step we used $\textup{ker}(\hat{V}_X) = G_X^\perp$ and
in the last step we used \emph{i)}.
 \end{proof}

 \begin{lemma}\label{lem:Conditioning X=Z}
Let \hyperref[Assumption:A]{Assumption A} be satisfied and $\mathcal{C} \subseteq \mathcal{A}$ be a sub-$\sigma$-algebra. Then for all $g \in L^2(\mathcal{C})$ the equality
     \begin{align}\label{eq:Conditioning X=Z}
         \int_\Omega g \mathbb{E}(X|\mathcal{C}) \, \textup{d} \mu = \int_\Omega g X \, \textup{d} \mu
     \end{align}
     is true. Moreover for all $w \in W_Y$ we have $L_W^* w \in W_Z$, where $L_W=\hat{V}_Y \hat{V}_X^*$ as in \eqref{eq:Definition of L_W as commutative diagramm}. 
 \end{lemma}

 \begin{proof}
 Equation \eqref{eq:Conditioning X=Z} follows by using \cite[Proposition 2.6.31]{martingal} with 
\begin{align*}
    \beta: \mathbb{R} \times E &\to E \\
    (a,e) &\mapsto ae.
\end{align*} 
To prove the second assertion we fix a $w\in W_Y$ and set $\mathcal{C}:=\sigma(Y)$. For $g:= V_Y^* w \in G_Y$, we then have $w=\int_\Omega g Y \, \textup{d} \mu$, which in turn leads to  
\begin{align*}
L_W^* w = \hat{V}_X \hat{V}_Y^* w =\int_\Omega g X \, \textup{d} \mu = \int_\Omega g Z \, \textup{d} \mu \in W_Z\, ,
\end{align*}    
where in the last equation we used \eqref{eq:Conditioning X=Z}.
\end{proof}

 \begin{lemma}\label{lem:L_W = L restricted}
Let \hyperref[Assumption:A]{Assumption A} be satisfied and $L:E \to F$ be a bounded operator. Moreover, let 
\begin{align*}
Y:= LX +N,
\end{align*}
with $N:\Omega \to F$ being a Gaussian random variable independent of $X$. Then we have
\begin{align*}
L_W = L_{|{W_X}}.
\end{align*}
\end{lemma}

\begin{proof}
Let us fix a $g \in G_X$, since $g$ is centered and $N$ is independent of $X$, we then have $\int_\Omega g N \, \textup{d} \mu =0$, and hence we obtain 
\begin{align*}
L_{|{W_X}} \left( \int_\Omega g X \, \textup{d} \mu \right) 
= \int_\Omega g LX \, \textup{d} \mu 
= \int_\Omega g (LX + N)\, \textup{d} \mu 
= \int_\Omega g Y \, \textup{d} \mu  
&= \hat{V}_Y g \\
&= L_W \left(\int_\Omega g X \, \textup{d} \mu \right)\, ,
\end{align*}
where in the last line we used the definition of $L_W$ as in \eqref{eq:Definition of L_W as commutative diagramm}, i.e. $L_W=\hat{V}_Y \hat{V}_X^*$.
\end{proof}

The next Lemma can be found for example in \cite[Proposition 2.6.31]{martingal}. 
\begin{lemma}\label{Lemma: e' commutes with conditioning}
    Let $E$ be a separable Banach space, $X \in L^1(\mathcal{A},E)$, and $\mathcal{B} \subseteq \mathcal{A}$ be a $\s$-algebra. Then for all $e'\in E'$ there exists an $N_{e'} \in \mathcal{A}$ with $\mu(N_{e'})=0$ such that
    \begin{align*}
        e'\left( \mathbb{E}(X|\mathcal{B}) \right) (\omega) = \mathbb{E}(e'(X)|\mathcal{B}) (\omega)
    \end{align*} 
    holds true for all $\omega \in \Omega \setminus N_{e'}$.
\end{lemma}

\begin{lemma}\label{lem:conditional variance schwache formulierung}
     Under \hyperref[Assumption:A]{Assumption A} the following equality holds true for all $e' \in E'$
     \begin{align*}
        \langle \textup{cov}(X|Y)e' , e' \rangle_{E,E'} = \textup{cov}(e'(X)|Y).
     \end{align*} 
\end{lemma} 
\begin{proof}
First we set $X_0 :=X-\mathbb{E}(X|Y)$.
We apply Lemma \ref{Lemma: e' commutes with conditioning} thrice to obtain 
\begin{align*}
\langle    \textup{cov}(X|Y) e' , e' \rangle_{E,E'} &= \left\langle e'\left(  \mathbb{E}((X-\mathbb{E}(X|Y)) \otimes (X-\mathbb{E}(X|Y)) | Y) \right) ,e' \right\rangle_{E,E'} \\
&=\left\langle e'\left(  \mathbb{E}( X_0 \otimes X_0 | Y) \right) ,e' \right\rangle_{E,E'} \\
&= \langle \mathbb{E}( e'(X_0) X_0 | Y) , e' \rangle_{E,E'} \\
&= e' \left( \mathbb{E}( e'(X_0) X_0) | Y) \right) \\
&= \mathbb{E}( e'(X_0) e'(X_0) | Y) \\
&= \mathbb{E}( (e'(X)-\mathbb{E}(e'(X)|Y))^2 | Y) \\
&= \textup{cov}(e'(X)|Y).
\end{align*}
\end{proof} 
The next theorem can be found in \cite[Theorem 3.3.2]{martingal}.
\begin{theorem}\label{martingal}
    Let $(\mathcal{A}_n)_{n\in \mathbb{N}}$ be a filtration in the probability space $(\Omega, \mathcal{A},\mu)$ and we define $\mathcal{A}_\infty:= \sigma\left( \mathcal{A}_n \, : \, n \geq 1 \right)$. Then for all $p \in [1,\infty)$ 
         and all $X \in L^p(\mathcal{A},E)$, we have  
        \begin{align*}
            \underset{n \to \infty}{\lim} \mathbb{E} (X | \mathcal{A}_n) = \mathbb{E}(X | \mathcal{A}_\infty )\, ,
        \end{align*} 
where the convergence is almost everywhere and in the norm of $L^p(\mathcal{A},E)$.
\end{theorem}

\begin{lemma}\label{lem:algebren gleich}
    Let \hyperref[Assumption:A]{Assumption A} be satisfied and $\mathcal{G}:=\{ f'(Y) \, | \, f' \in F' \}$. Then we have $ \sigma(Y)=\sigma(\mathcal{G}) $.
\end{lemma}

\begin{proof}
For $f' \in F'$, we have 
\begin{align*}
    Y^{-1} (\sigma(f')) = Y^{-1} (f'^{-1}(\mathcal{B}))= (f' \circ Y)^{-1} (\mathcal{B}) = \sigma (f' (Y)) \, , 
\end{align*}
where $\mathcal{B}$ is the Borel-$\sigma$-algebra of $\mathbb{R}$. 
This leads to
\begin{align*}
\s(Y)
=
Y^{-1}(\sigma(F'))
= Y^{-1} \left(\s \left(\bigcup_{f'\in F'} \s(f') \right) \right)
&= \s\left(Y^{-1} \left( \bigcup_{f'\in F'} \s(f') \right) \right) \\
&= \s \left(\bigcup_{f'\in F'}  Y^{-1} ( \s(f')) \right) \\
&=  \s\left(\bigcup_{f'\in F'}  \s(f' ( Y) ) \right) \\
&= \s\left(\bigcup_{g\in \mathcal{G} }\s(g) \right)\\
&= \s(\mathcal G)\, ,
\end{align*} 
where in the third step we used \cite[Theorem 1.81]{klenke2020probability}. 
\end{proof}

\begin{lemma}\label{lem:Hahn-Banach}
    Let $E$ be a Banach space and $\mathcal{D} \subset B_{E'}$ a weak-$*$-dense subset. Then for all $e_1,e_2 \in E$ with $e'(e_1)=e'(e_2)$ for all $e' \in \mathcal{D}$, we have $e_1=e_2$.
\end{lemma}
\begin{proof}
Let $e' \in B_{E'}$ then there exists a net $e_\alpha' \subseteq \mathcal{D}$ with $ \langle e_\alpha', e_j \rangle_{E',E} \to \langle e', e_j \rangle_{E',E} $ for $j=1,2$. This shows by our assumption that $e'(e_1)=e'(e_2)$. Now a simple application of Hahn-Banach Theorem gives the assertion. 
\end{proof}
The next Lemma can be found in \cite[Theorem 2.6.18 and 2.6.23]{megginson2012introduction}. 
\begin{lemma}\label{lem:countable pointwise dense subset}
Let $E$ be a separable Banach space. Then $B_{E'}$ endowed with the weak-$*$-topology is a compact metrizable space and there exists a countable, weak-$*$-dense subset $\mathcal{D} \subseteq B_{E'}$.
\end{lemma} 

\begin{lemma}\label{lem:Hahn-Banach für Zufallsvariablen}
Let $(\Omega,\mathcal{A},\mu)$ be a probability space, $E$ be a separable Banach space, and $X,Y: \Omega \to E$ random variables.
Moreover, assume that for all $e' \in E'$ there exists a null-set $N_{e'} \in \mathcal{A}$ such that 
\begin{align*}
e'(X(\omega))= e'(Y(\omega)), \qquad \, \omega \in \Omega \setminus N_{e'}\, .
\end{align*}
Then $X=Y$ almost everywhere.
\end{lemma}
\begin{proof}
Let $\mathcal{D}\subset B_{E'}$ be a countable weak-$*$-dense subset, according to Lemma \ref{lem:countable pointwise dense subset}. We set 
\begin{align*}
N_{\mathcal{D}} := \cup_{e' \in \mathcal{D}} N_{e'}.
\end{align*}
Then $N_\mathcal{D}$ is a union of countable null-sets, thus again a null-set.
Moreover, our construction  shows 
\begin{align*}
e'(X(\omega))=e'(Y(\omega)), \qquad \, \omega \in \Omega \setminus N_\mathcal{D},\, e' \in \mathcal{D}.
\end{align*}
Applying Lemma \ref{lem:Hahn-Banach} leads to $X(\omega) = Y(\omega)$ for all $\omega \in \Omega \setminus N_\mathcal{D}$.
\end{proof}

\begin{lemma}\label{lem:Projektions Darstellungen}
Under \hyperref[Assumption:A]{Assumption A}, let $G$ be a Gaussian Hilbert space, such that $G_X \subseteq G$. Additionally let $\mathcal{G} \subseteq G$ be a closed subspace and $(r_j)_{j \in \mathbb{N}}$ be an ONB of $\mathcal{G}$. Then the series
\begin{align*}
\Pi_\mathcal{G} X:=\sum_{j = 1}^\infty r_j \int_\Omega r_j X \, \textup{d} \mu
\end{align*}
converges almost everywhere as well as in $L^p(\mathcal{A},E)$
for all $p \geq 1$ and its limit $\Pi_\mathcal{G} X$ is jointly Gaussian with $X$. Moreover the limit is independent of the choice of the ONB, that is if $(\tilde{r}_j)_{ j\in \mathbb{N}}$ is another ONB of $\mathcal{G}$, then  
\begin{align*}
\sum_{j =1}^\infty r_j \int_\Omega r_j X \, \textup{d} \mu = \sum_{j = 1}^\infty \tilde{r}_j \int_\Omega \tilde{r}_j X \, \textup{d} \mu
\end{align*}
almost everywhere. Finally, if $\Pi_{\mathcal G}:G\to G$ denotes the orthogonal projection onto $\mathcal G$, then for all $e' \in E'$ we   have  
\begin{align}\label{eq:e' kommutiert mit Pi}
    e'(\Pi_\mathcal{G} X) = \Pi_\mathcal{G} e'(X) \, .
\end{align}
\end{lemma}

\begin{proof}
In order to apply Theorem \ref{martingal}, we define the filtration given by the $\sigma$-algebras
\begin{align*}
\mathcal{A}_n:= \sigma \left( r_1, \dots, r_n  \right).
\end{align*} 
For $e' \in E'$, Lemma \ref{Lemma: e' commutes with conditioning} and \eqref{eq:Conditioning Orthogonal projection} applied to $V:=\textup{span} \{ r_1, \dots ,r_n \}$ give
\begin{align*}
e'\left( \mathbb{E} (X|\mathcal{A}_n) (\omega)\right)=\mathbb{E}(e'(X)|\mathcal{A}_n)(\omega) = \sum_{j=1}^n r_j(\omega) \int_\Omega r_j e'(X) \, \textup{d} \mu = e' \left( \sum_{j=1}^n r_j(\omega) \int_\Omega r_j X \, \textup{d} \mu \right)
\end{align*}
for all $\omega \in \Omega \setminus N_{e'}$. Consequently,  Lemma \ref{lem:Hahn-Banach für Zufallsvariablen} shows
\begin{align*}
\mathbb{E}(X| \mathcal{A}_n)  =  \sum_{j=1}^n r_j \int_\Omega r_j X \, \textup{d} \mu.
\end{align*} 
Moreover, using Theorem \ref{martingal}, we obtain the convergence
\begin{align*}
    \mathbb{E}(X| \mathcal{A}_n) \to \mathbb{E}(X | \mathcal{A}_\infty),
\end{align*}
which occurs almost everywhere and in $L^p(\mathcal{A},E)$. This implies that the series
\begin{align*}
   \sum_{j=1}^\infty r_j \int_\Omega r_j X \textup{d} \mu =  \mathbb{E}(X | \mathcal{A}_\infty)
\end{align*}
is convergent in the same sense. 
In order to show that the resulting random variable is Gaussian we consider the sequence
\begin{align*}
    Z_n:= \mathbb{E}(X| \mathcal{A}_n)  =\sum_{j=1}^n r_j \int_\Omega r_j X \, \textup{d} \mu .
\end{align*} 
Each $Z_n$ is an $E$-valued Gaussian random variable and, as noted above, the sequence $(Z_n)$ converges to the random variable $\Pi_{\mathcal{G}} X= \mathbb{E}(X|\mathcal{A}_\infty)$ pointwise on the set $\Om\setminus N$, where $N$ is a suitable null-set.
Consequently, we have $e'(Z_n(\om)) \to e'(\Pi_{\mathcal{G}}(\om))$ for all $e'\in E'$ and $\om \in \Om\setminus N$. 

We note for $e_1',e_2' \in E'$ we have $e'_1(\Pi_{\mathcal{G}} X) \in \mathcal{G} \subseteq G$ and $e_2'(X) \in G_X \subseteq G$, and thus $e_1'(\Pi_{\mathcal{G}}X)+e_2'(X)$ is as an element of $G$ Gaussian. We conclude $\Pi_{\mathcal{G}}X$ is jointly Gaussian with $X$.

Moreover, since each $e'(Z_n)$ is an $\R$-valued Gaussian random variable, so is the limit $e'(Z_\infty)$, 
see Lemma \ref{lem:LemmaC3Steinwart}. 

To verify the last assertion we first note that
applying   \eqref{eq:Conditioning Orthogonal projection}   to $V := \mathcal{G}$ and $g:= e'(X)$ yields
\begin{align}\label{eq:Equality for e' and Pi weak}
 \sum_{j =1}^\infty r_j \int_\Omega r_j e'( X ) \, \textup{d} \mu = \Pi_{\mathcal{G}} e'(X)\, ,
\end{align}
where both sides are $L^2({\mathcal A})$ $\mu$-equivalence classes.
Moreover, the series 
\begin{align*}
     \sum_{j = 1}^\infty r_j \int_\Omega r_j X \, \textup{d} \mu  
\end{align*}
converges pointwise up to some null set $N$. Consequently, we 
find
\begin{align*}
e'(\Pi_\mathcal{G} X(\omega))
=
e'\left( \sum_{j = 1}^\infty r_j(\omega) \int_\Omega r_j X \, \textup{d} \mu  \right) 
=
\sum_{j =1}^\infty r_j(\omega) \int_\Omega r_j e'( X ) \, \textup{d} \mu 
\end{align*}
for all $\om \in \Om \setminus N$.
Combining   this with \eqref{eq:Equality for e' and Pi weak} yields \eqref{eq:e' kommutiert mit Pi}.

Finally, let $(\tilde{r}_j)_{ j\in \mathbb{N}}$ be  another ONB of $\mathcal{G}$. The already proven parts then show that 
\begin{align*}
\tilde \Pi_\mathcal{G} X:=\sum_{j = 1}^\infty \tilde r_j \int_\Omega \tilde r_j X \, \textup{d} \mu
\end{align*}
converges almost everywhere and for all $e'\in E'$ we have 
$e'(\tilde \Pi_\mathcal{G} X) = \Pi_\mathcal{G} e'(X) = e'(\Pi_\mathcal{G} X)$ as $L^2({\mathcal A})$ $\mu$-equivalence classes.
Now Lemma \ref{lem:Hahn-Banach für Zufallsvariablen} leads to the desired result.
\end{proof}


\begin{lemma}\label{lem:Orthogonal projections are similiar}
Let $H_1,H_2$ be Hilbert spaces, $\Pi_1: H_1 \to H_1$ be an orthogonal projection, and $U:H_1 \to H_2$ be an isometric isomorphism. Then 
$\Pi_2:=U \Pi_1 U^*$ is an orthogonal projection in $H_2$ with $U(\textup{ran}( \Pi_1))=\textup{ran}(\Pi_2) $.
\end{lemma}

\begin{proof} 
Using \cite[Proposition II.3.3]{conway1990functional}, we note that we only have to prove that $\Pi_2$ is idempotent and self adjoint.
We first note that $\Pi_2$ is idempotent since
\begin{align*}
\Pi_2^2   = U \Pi_1 U^* U \Pi_1 U^*   = U \Pi_1^2 U^*   = U \Pi_1 U^*  = \Pi_2\, .
\end{align*} 
Moreover, we also have 
\begin{align*}
\Pi_2^*  = (U \Pi_1 U^*)^*  = ((U^*)^* \Pi_1^* U^*)  = U \Pi_1 U^*  = \Pi_2\, ,
\end{align*}
i.e.~$\Pi_2$ is self-adjoint. The last assertion is trivial. 
\end{proof}

\begin{lemma}\label{lem:orthogonal bochner integral}
Let \hyperref[Assumption:A]{Assumption A} be satisfied. Then for all $g \in L^2(\mathcal{A})$ we have
\begin{align*}
    \int_\Omega g X \, \textup{d}\mu = \int_\Omega \Pi_{G_X} g X \, \textup{d} \mu\, ,
\end{align*}
where $\Pi_{G_X}:L^2(\mathcal{A}) \to L^2(\mathcal{A})$ is the orthogonal projection onto $G_X$.
\end{lemma}
\begin{proof}
    Since $G_X$ is a closed subspace of $L^2(\mathcal{A})$, we decompose $g$ into $g= \Pi_{G_X} g + \Pi_{G_X^\perp}g$, with $\Pi_{G_X^\perp}$ being the orthogonal projection onto the orthogonal complement of $G_X$. For simplicity we set $g_X:= \Pi_{G_X}g$ and $g_X^\perp=\Pi_{G_X^\perp} g$. Given $e' \in E'$, we then have
    \begin{align*}
e'\left(        \int_\Omega g X \, \textup{d} \mu \right) =       \int_\Omega g e'(X) \, \textup{d} \mu =         \int_\Omega (g_X+g_X^\perp) e'(X) \, \textup{d} \mu 
&=         \int_\Omega g_X e'(X) \, \textup{d} \mu \\
&=e'\left(         \int_\Omega g_X X \, \textup{d} \mu \right),
    \end{align*}
    where we used $e'(X) \in G_X$ and $g_X^\perp \in G_X^\perp$. Since this holds true for all $e' \in E'$ we find the assertion by the Hahn-Banach theorem.
\end{proof}

\begin{lemma}\label{lem:Subsets}
Under \hyperref[Assumption:A]{Assumption A} it holds $G_Z \subseteq G_Y=G_{(Z,Y)}$ and if $G_Y \subseteq G_X$, we additionally have $G_Y=G_Z$.
\end{lemma} 
 \begin{proof}
    For $G_Z \subseteq G_Y$, see Theorem \ref{the:Theorem3.3Steinwart}. We conclude $G_Y=G_{(Z,Y)}$ by 
    \begin{align*}
G_Y \subseteq    G_{(Z,Y)} = \overline{G_Z + G_Y} \subseteq \overline{G_Y + G_Y} = G_Y\, ,
\end{align*}
where in the first and second step we used \eqref{eq:G_{(X,Y)}}.

To prove the second assertion, we fix a $g \in G_Y$. Since $G_Y\subseteq G_X$,  there then exists a sequence $(e_n')$ in $E'$ such that $e'_n(X) \to g$ in $L^2(\mathcal{A})$. Now Lemma \ref{lem:Conditioning X=Z} with $\mathcal{C}:=\sigma(Y)$ implies for all $g_Y \in G_Y \subseteq L^2(\mathcal{C})$
\begin{align*}
 \int_\Omega g_Y e_n'(Z) \, \textup{d} \mu =    \int_\Omega g_Y e'_n(X) \, \textup{d} \mu \to \int_\Omega g_Y g \, \textup{d} \mu \, . 
\end{align*} 
We conclude that the sequence $(e_n'(Z))$ converges $L^2({\mathcal C})$-weakly against $g$. Moreover, using Theorem \ref{the:Theorem3.3Steinwart}, we find
\begin{align*}
\Vert    e_n'(Z) -e_m'(Z) \Vert_{L^2(\mathcal{C})} = \Vert  \Pi_{G_Y} ( e_n'(X) -e_m'(X) )\Vert_{L^2(\mathcal{C})} \leq \Vert    e_n'(X) -e_m'(X) \Vert_{L^2(\mathcal{A})}\, .
\end{align*}
Since $(e_n'(X))$ is a  $L^2({\mathcal A})$-Cauchy sequence, we conclude that $(e_n'(Z))$ is a  $L^2({\mathcal C})$-Cauchy sequence.
Combining the latter with the weak convergence, we conclude that $e_n'(Z) \to g$ in $L^2(\mathcal{C})$. In other words we have found $g \in G_Z$.
\end{proof}

\begin{lemma}
    Under \hyperref[Assumption:A]{Assumption A}, it holds $W_Z \subseteq W_X$ and for all $w \in W_Z$, we have $\Vert w \Vert_{W_X} \leq \Vert w \Vert_{W_Z}$.
\end{lemma}

\begin{proof}
By Lemma \ref{lem:Conditioning X=Z}  with $\mathcal{C}: = \s(Y)$ and Lemma \ref{lem:Unitar} 
we find 
\begin{align*}
     W_{Z}
     =  \left\{ f \in E \, \middle| \, \exists g \in G_Z, f = \int_\Omega g X \, \textup{d} \mu \right\} 
     \subseteq
      \left\{ f \in E \, \middle| \, \exists g \in L^2(\mathcal{A}), f = \int_\Omega gX  \, \textup{d} \mu \right\}
      = W_X\, .
\end{align*}
For the proof of the second assertion we fix a 
$w \in W_Z$.  We then obtain
\begin{align*}
    \Vert w \Vert_{W_Z} = \inf_{g \in G_Z \, w = V_Z g} \Vert g \Vert_{L^2(\mathcal{A})} = \inf_{g \in G_Z \, w = \hat{V}_X g} \Vert g \Vert_{L^2(\mathcal{A})} \geq \inf_{g \in L^2(\mathcal{A}) \, w = \hat{V}_X g} \Vert g \Vert_{L^2(\mathcal{A})} = \Vert w \Vert_{W_X}\, ,
\end{align*}
where in the second step we used Lemma \ref{lem:Conditioning X=Z} and in the last step we used Lemma \ref{lem:Unitar}.
\end{proof}

\begin{proof}[Proof of Theorem \ref{Theorem:Conditioning orthogonal projection}]

Applying Lemma \ref{Lemma: e' commutes with conditioning}, Lemma \ref{lem:algebren gleich}, Identity \eqref{eq:Conditioning Orthogonal projection}, and Lemma \ref{lem:Projektions Darstellungen} leads to 
\begin{align*} 
e'\left(\mathbb{E}(X|Y) \right)=\mathbb{E}(e'(X)| Y) =    \mathbb{E}(e'(X)| G_Y) = \Pi_{G_Y} e'(X)=   e'\left( \Pi_{G_Y} X \right) \, ,
\end{align*} 
where $\Pi_{G_Y}$ denotes the orthogonal projection $\Pi_{G_Y}:G_{(X,Y)} \to G_{(X,Y)}$ onto $G_Y$.
Using Lemma \ref{lem:Hahn-Banach für Zufallsvariablen} and Lemma \ref{lem:Projektions Darstellungen}, we obtain
\begin{align*}
    \mathbb{E}(X|Y) = \Pi_{G_Y}X = \sum_{j \in J} r_j \int_\Omega r_j X \, \textup{d} \mu,
\end{align*}
where the series converges almost everywhere and in $L^p(\mathcal{A},E)$. 

\emph{i).} Since $\mathbb{E}(\mycdot|Y):G_{(X,Y)} \to G_{(X,Y)}$ is the orthogonal projection onto $G_Y$, see Theorem \ref{the:Theorem3.3Steinwart}, we conclude $\Pi_{G_Y}=\mathbb{E}(\mycdot|Y)$. We have $G_Y=G_{(Z,Y)}$ with Lemma \ref{lem:Subsets}.
We obtain $\Pi_G=\Pi_{G_Y} = \mathbb{E}(\mycdot | Y)$.

\emph{ii).} We consider the map $\Pi_W:W_{(X,Y)} \to W_{(X,Y)}$ that is defined by 
\begin{align*}
    \Pi_W (w_x,w_y) = (L^*_W w_y, w_y),
\end{align*}
and show that it is the orthogonal projection onto $W_{(Z,Y)}$.
To this end, let $ g \in G_{(X,Y)}$ and $g_Y:=\Pi_G g$.
Then we find
\begin{align*}
V_{(X,Y)}^* \Pi_W V_{(X,Y)} g &= V_{(X,Y)}^* \Pi_W \int_\Omega g \cdot (X,Y) \, \textup{d} \mu \\
&= V_{(X,Y)}^* \left(L_W^* \int_\Omega g Y \, \textup{d} \mu, \int_\Omega g Y \, \textup{d} \mu \right) \\
&= V_{(X,Y)}^* \left(L_W^* \int_\Omega g_Y Y \, \textup{d} \mu, \int_\Omega g_Y Y \, \textup{d} \mu \right) \\
&= V_{(X,Y)}^* \left( \int_\Omega g_Y X \, \textup{d} \mu, \int_\Omega g_Y Y \, \textup{d} \mu \right)  \\
&= g_Y \, ,
\end{align*}
where the third identity is a consequence of  Lemma \ref{lem:orthogonal bochner integral} and the fourth identity follows by the definition of $L_W$, see \eqref{eq:Definition of L_W as commutative diagramm}. We use Lemma \ref{lem:Orthogonal projections are similiar} to conclude that $\Pi_W$ is an orthogonal projection onto the space
\begin{align*}
 V_{(X,Y)} \textup{ran}(\Pi_{G}) =V_{(X,Y)} G_{Y}  &= \left\{ \int_\Omega g \cdot (X,Y) \,\textup{d} \mu  \, \middle| \, g \in G_Y \right\} \\
 &=\left\{ \int_\Omega g \cdot (Z,Y) \,\textup{d} \mu  \, \middle| \, g \in G_{Y} \right\} \\
 &=\left\{ \int_\Omega g \cdot (Z,Y) \,\textup{d} \mu  \, \middle| \, g \in G_{(Z,Y)} \right\} \\
 &=W_{(Z,Y)}\, ,
\end{align*}
where in the third identity we used Lemma \ref{lem:Conditioning X=Z}, with $\mathcal{C}=\sigma(Y)$ and in the fourth identity we used Lemma \ref{lem:Subsets}.

\emph{iii).} We define the mapping $\Pi_H: H_{(X,Y)} \to H_{(X,Y)}$ by
\begin{align*}
    \Pi_H (h_x,h_y) := (\iota_X L_W^* \iota_Y^{-1} h_y,h_y) 
\end{align*} 
and show that this is the orthogonal projection onto $H_{(Z,Y)}$. 
To this end we first note that for $(h_X,h_Y) \in H_{(X,Y)}$ and $( w_X,  w_Y):=( \iota_X^{-1} h_X, \iota_Y^{-1} h_Y) \in W_{(X,Y)}$, we have  
\begin{align*}
    \iota_{(X,Y)} \Pi_W \iota_{(X,Y)}^{-1} (h_X,h_Y) = \iota_{(X,Y)} \Pi_W (w_X,w_Y) = \iota_{(X,Y)} (L_W^* w_Y,w_Y) = (\iota_X L_W^* \iota_Y^{-1} h_Y,h_Y)\, .
\end{align*}
Moreover recall from \eqref{eq: U_X} that we have
\begin{align*}
    U_{(X,Y)}= V_{(X,Y)}^{-1} \iota_{(X,Y)}^{*}.
\end{align*} 
Combining both identities leads to
\begin{align*}
    U_{(X,Y)} \Pi_H U^*_{(X,Y)} &= V_{(X,Y)}^{*} \iota_{(X,Y)}^{-1} \circ \Pi_H \circ \iota_{(X,Y)} V_{(X,Y)} \\
    &= V_{(X,Y)}^{*} \iota_{(X,Y)}^{-1} \circ     \iota_{(X,Y)} \circ \Pi_W \circ \iota_{(X,Y)}^{-1} \circ \iota_{(X,Y)} V_{(X,Y)} \\
    &= V_{(X,Y)}^* \Pi_W V_{(X,Y)} \, .
\end{align*}
By Lemma \ref{lem:Orthogonal projections are similiar} we conclude that $\Pi_H$ is an orthogonal projection with
\begin{align*} 
\textup{ran}(\Pi_H) = \iota_{(X,Y)} \textup{ran}(\Pi_W) =
\left\{ \iota_{(X,Y)} \int_\Omega g \cdot (Z,Y) \,\textup{d} \mu  \, \middle| \, g \in G_{(Z,Y)} \right\} = H_{(Z,Y)}\, ,
\end{align*}
where in the last equality we used 
\begin{align*}
    \iota_X \left( \int_\Omega g Z \, \textup{d} \mu \right) = \left(e' \mapsto \int_\Omega g e'(Z) \, \textup{d} \mu  \middle)\right|_{B_{E'}} = \iota_Z \left( \int_\Omega g Z\, \textup{d} \mu \right)\, .
\end{align*}
\end{proof}

We define the cross covariance operator $\text{cov}(X,Y):F' \to E$ of $X$ and $Y$ as 
\begin{align*}
 \textup{cov}(X,Y)    f':= \int_{\Omega} f'(Y) X \, \textup{d} \mu.
\end{align*} 

\begin{lemma}\label{lem:Covariance Rechenregel}
Let \hyperref[Assumption:A]{Assumption A} be satisfied, $\mathcal{C} \subseteq \mathcal{A}$ be a sub-$\sigma$-algebra, and  $X_\mathcal{C}:=\mathbb{E}(X|\mathcal{C})$. Then for all $e_1',e_2 ' \in E'$ the following identities hold true 
 
\begin{enumerate}[label={\roman*)}]
\item\label{item:1}$\textup{cov}(X,X_\mathcal{C})= \textup{cov}(X_\mathcal{C},X)=\textup{cov}(X_\mathcal{C}),$
\item\label{item:2}$ \textup{cov}(X-X_\mathcal{C}) =  \textup{cov}(X) -  \textup{cov}(X_\mathcal{C}).$
\end{enumerate}

\end{lemma}

\begin{proof}
We start with $\ref{item:1}$.
Taking $e_1',e_2' \in E'$, we obtain with Lemma \ref{lem:Conditioning X=Z} and the dual pairing of $\langle \mycdot, \mycdot \rangle_{E,E'}$ 
\begin{align*}
\langle \textup{cov}(X_\mathcal{C},X)e'_1,e_2' \rangle_{E,E'} 
= \int_\Omega e_1'(X_\mathcal{C})e'_2(X) \, \textup{d} \mu 
= \int_\Omega e_1'(X_\mathcal{C})e'_2(X_\mathcal{C}) \, \textup{d} \mu 
= \langle \textup{cov}(X_\mathcal{C})e'_1,e_2' \rangle_{E,E'}\, .
\end{align*} 
The identity $\textup{cov}(X, X_\mathcal{C})=\textup{cov}(X_\mathcal{C})$ can be shown analogously. 

The second assertion follows from the bilinearity of the cross covariance operator, namely  
\begin{align*}
 \textup{cov}(X-X_\mathcal{C})
&= 
  \textup{cov}(X-X_\mathcal{C}, X-X_\mathcal{C}) \\
&= 
 \textup{cov}(X,X) - \textup{cov}(X, X_\mathcal{C}) -  \textup{cov}(X_\mathcal{C}, X) +  \textup{cov}(X_\mathcal{C}, X_\mathcal{C}) \\
&= \textup{cov}(X,X) - \textup{cov}(X_\mathcal{C}, X_\mathcal{C})  ,
\end{align*}
where in the last step we used the first assertion.
\end{proof}

In the following, we like to introduce conditional expectations $\mathbb{E}(X|G_Y)$ with respect to a given Gaussian Hilbert space $G_Y$. Here we note that formally $G_Y$ is a subset of $L^2(\mathcal{A})$ that is, a collection of \emph{$\mu$-equivalence classes}. Consequently, a definition like 
\begin{align*}
\mathbb{E}(X|G_Y) := \mathbb{E} \bigl(  X \, \bigr| \, \sigma \left( \left\{  g : g \in G_Y \right\} \right)  \bigr)
\end{align*} 
is \emph{not} well defined. To address this problem we need a couple of auxiliary results. We begin with the following characterization of conditional expectations, where in its proof we need the indicator function $\chi_A: \Omega \to \mathbb{R}$ of $A\subseteq \Omega$, that is, the function defined by
\begin{align*}
\chi_A (\omega) := \begin{cases}
1, &\text{if } \omega \in A \\
0, &\text{else. } 
\end{cases}
\end{align*}

\begin{lemma}\label{lem:BedingteErwartungErzeugendeSystem}
Let $(\Omega,\mathcal{A}, \mu)$ be a probability space, $X \in L^1(\mathcal{A},E)$, and $\mathcal{B} \subset \mathcal{A}$ be a sub-$\sigma$-algebra. Furthermore let $\mathcal{E} \subset \mathcal{A}$ be $\cap$-stable system with $\Omega \in \mathcal{E}$ and $\sigma(\mathcal{E})=\mathcal{B}$. If $Z:\Omega \to E$ is $\mathcal{B}$-measurable, $\mu$-integrable, and satisfies 
\begin{align}\label{eq:Conditionexpectaion for Erzeugende}
\int_B Z \, \textup{d}\mu = \int_B X \, \textup{d} \mu 
\end{align}
for all $B \in \mathcal{E}$, then we have $Z \in \mathbb{E}(X|B)$.  
\end{lemma} 
\begin{proof}
We consider the Banach space valued mappings $Q,P : \mathcal{B} \to E$ given by 
\begin{align*}
Q(B) := \int_B X \, \textup{d} \mu, \\
P(B) := \int_B Z \, \textup{d} \mu.
\end{align*}
Moreover we define 
\begin{align*}
\mathcal{D} := \{ B \in \mathcal{B} \, | \, Q(B) =P (B) \}.
\end{align*} 
Note that $\Omega \in \mathcal{E}$ by assumption, and \eqref{eq:Conditionexpectaion for Erzeugende} implies $\mathcal{E} \subseteq \mathcal{D}$. 
Let us now show that $\mathcal{D}$ is a Dynkin system in the sense of \cite[Chapter 1.6]{cohn2013measure}. Obviously we have $\Omega \in \mathcal{E}  \subseteq \mathcal{D}$. 
Moreover, if $A,B \in \mathcal{D}$ with $A \subseteq B$ then 
\begin{align*}
\int_{B\setminus A} Z \, \textup{d} \mu = \int_B Z \, \textup{d} \mu - \int_A Z \, \textup{d} \mu = \int_B X \, \textup{d} \mu - \int_A X \, \textup{d} \mu = \int_{B \setminus A} X \, \textup{d} \mu 
\end{align*} 
and therefore $B \setminus A \in \mathcal{D}$. 
Finally if we have an increasing sequence $(A_n) \subseteq \mathcal{D}$ then for $A := \cup_{n \in \mathbb{N}} A_n$ we have 
\begin{align*}
\int_A Z \, \textup{d} \mu = \int_\Omega \lim_{n \to \infty} \chi_{A_n} Z \, \textup{d} \mu = \lim_{n \to \infty} \int_{A_n} Z \, \textup{d} \mu = \lim_{n \to \infty} \int_{A_n} X \, \textup{d} \mu = \int_A X \, \textup{d} \mu,
\end{align*}
where we used in the second and last step the dominated convergence theorem for Bochner integrals \cite[Proposition 1.2.5]{martingal}.
Therefore $\mathcal{D}$ is indeed Dynkin class. By \cite[Theorem 1.6.2]{cohn2013measure} we find that $\sigma(\mathcal{E})= d(\mathcal{E})\subseteq d(\mathcal{D})=\mathcal{D}$, where $d(\mathcal{E})$ and $d(\mathcal{D})$ denote the smallest Dynkin system containing $\mathcal{E}$ and $\mathcal{D}$, respectively. Since $\sigma(\mathcal{E})=\mathcal{B}$ we find the assertion.
\end{proof}

Let us now consider the $\sigma$-algebra $\mathcal{T}$ of $\mu$-trivial events, that is
\begin{align*}
\mathcal{T}:=\{ T \in \mathcal{A} \, | \, \mu(T) \in \{0,1\} \}.
\end{align*}
Given a $\sigma$-algebra $\mathcal{B}\subseteq \mathcal{A} $ we further define 
\begin{align*}
\mathcal{B}_\cap &{:=} \{ B \cap T \, | \, B \in \mathcal{B}, \, T \in \mathcal{T} \}, \\
\mathcal{B}_\cup &{:=} \{ B \cup T \, | \, B \in \mathcal{B}, \, T \in \mathcal{T} \}.
\end{align*}
The next lemma collects some simple but yet important properties of these sets.

\begin{lemma}\label{lem:sigma B schnitt = hut B}
For all $\s$-algebras $\mathcal{B} \subseteq \mathcal{A}$ the following statements hold true:
\begin{enumerate}[label={\roman*)}]
\item\label{item:1b} The set $\mathcal{B}_\cap$ is $\cap$-stable and $\Omega \in \mathcal{B}_\cap$. 
\item\label{item:2b} We have $\sigma(\mathcal{B}_\cap) = \sigma(\mathcal{B}_\cup)=:\hat{\mathcal{B}}$.
\item\label{item:3b} It holds that $\mathcal{B} \subseteq \hat{\mathcal{B}}$.
\item\label{item:4b} We have $\hat{\mathcal{B}} = \hat{\hat{\mathcal{B}}}$.
\item\label{item:5b} Given a $\s$-algebra $\mathcal{C} \subseteq \hat{\mathcal{B}}$ we have $\hat{\mathcal{C}} \subseteq \hat{\mathcal{B}}$.
\end{enumerate} 
\end{lemma} 
\begin{proof}
$\ref{item:1b}$. The first statement follows by the fact that intersections are commutative and associative, and the second statement is obvious.

$\ref{item:2b}$. We first show that $\mathcal{B}_\cap \subseteq \sigma (B_\cup)$. To this end let $B \in \mathcal{B}$ and $T \in \mathcal{T}$. Then we have 
\begin{align*}
B \cap T = \Omega \setminus (\Omega \setminus (B \cap T)) = \Omega \setminus \left( ( \Omega \setminus B) \cup (\Omega \setminus T ) \right) \in \sigma(\mathcal{B}_\cup).
\end{align*}
To prove $\mathcal{B}_\cup\subseteq \sigma(\mathcal{B}_\cap)$ we fix again some $B \in \mathcal{B}$ and $T \in \mathcal{T}$. Then we have
\begin{align*}
B \cup T = \Omega \setminus ( \Omega \setminus (B \cup T)) = \Omega \setminus \left( ( \Omega \setminus B ) \cap ( \Omega \setminus T ) \right)) \in \sigma(\mathcal{B}_\cap). 
\end{align*}
In summary we thus find $\sigma(\mathcal{B}_\cap) \subseteq \sigma(\sigma(\mathcal{B}_\cup))= \sigma(\mathcal{B}_\cup) \subseteq \sigma(\sigma(\mathcal{B}_\cap))=\sigma(\mathcal{B}_\cap)$.

$\ref{item:3b}$. This follows from $\mathcal{B} \subseteq \mathcal{B}_\cap$. 

$\ref{item:4b}$. 
For the first statement we note that $\hat{\mathcal{B}} \subseteq \hat{\hat{\mathcal{B}}}$ is obvious. For the converse inclusion we note that $\hat{\mathcal{B}}_\cap \subseteq \hat{\mathcal{B}}$ holds true. 
Since $\hat{\mathcal{B}}$ is a $\s$-algebra this statement follows.

\ref{item:5b}.
This follows by $\hat{\mathcal{C}} \subseteq \hat{\hat{\mathcal{B}}}=\hat{\mathcal{B}}$.
\end{proof}

\begin{lemma}
Given $\sigma$-algebras $\mathcal{B} \subseteq \mathcal{C} \subseteq \hat{\mathcal{B}} \subseteq \mathcal{A}$, we have $\hat{\mathcal{C}}=\hat{\mathcal{B}}$.
\end{lemma}

\begin{proof}
We first show $\hat{\mathcal{C}} \subseteq \hat{\mathcal{B}}$. To this end let $C \in \mathcal{C}$ and $T \in \mathcal{T}$. By assumption we find $C  \in \hat{\mathcal{B}}$ and $T \in \mathcal{T} \subseteq \hat{\mathcal{B}}$, and thus $C \cap T \in \hat{\mathcal{B}}$. This shows $\hat{\mathcal{C}}= \sigma(\mathcal{C}_\cap) \subseteq \hat{\mathcal{B}}$.

For the converse conclusion we note that $\mathcal{B} \subseteq \mathcal{C}$ implies $\mathcal{B}_\cap \subseteq \mathcal{C}_{\cap}$. By Lemma \ref{lem:sigma B schnitt = hut B} we conclude $\hat{\mathcal{B}} = \sigma(\mathcal{B}_\cap) \subseteq \sigma(\mathcal{C}_{\cap}) = \hat{\mathcal{C}}$.
\end{proof}

\begin{lemma}\label{lem:Z bis auf Nullmengen}
Let $X \in L^1(\mathcal{A}, E)$. Then for all $\sigma$-algebras $\mathcal{B} \subseteq \mathcal{C} \subseteq \hat{\mathcal{B}} \subseteq \mathcal{A}$, we have
\begin{align*}
\mathbb{E}(X|\mathcal{B}) \subseteq \mathbb{E}(X|\mathcal{C}) \subseteq \mathbb{E}(X|\hat{\mathcal{B}}).
\end{align*}
Furthermore, for all $Z \in \mathbb{E}(X|\mathcal{B})$, $Z_\mathcal{C} \in \mathbb{E}(X|\mathcal{C})$, and $\hat{Z} \in \mathbb{E}(X|\hat{\mathcal{B}})$, we have
\begin{align}\label{eq:Z bis auf Nullmengen}
\mu(Z \neq Z_\mathcal{C}) = \mu(Z \neq \hat{Z}) = \mu(Z_\mathcal{C} \neq \hat{Z}) = 0.
\end{align}
\end{lemma}

\begin{proof}
Let us fix a $Z \in \mathbb{E}(X|\mathcal{B})$. Our first goal is to show $Z \in \mathbb{E}(X|\hat{\mathcal{B}})$. Then $Z$ is $\hat{\mathcal{B}}$-measurable. Moreover, for $B \in \mathcal{B}$ and $T \in \mathcal{T}$ with $\mu(T)=1$, we have
\begin{align}\label{eq:B cap T conditional expectation}
\int_{B \cap T} Z \, \textup{d} \mu &= \int_B Z \, \textup{d} \mu = \int_B X \, \textup{d} \mu = \int_{B \cap T} X \, \textup{d} \mu,
\end{align}
where in the first and last step we used $B = (B \cap T) \cup (B \cap T^c)$ and  $\mu(B \cap T^c) \leq \mu(T^c) = 0$.
Furthermore, for $B \in \mathcal{B}$ and $T \in \mathcal{T}$ with $\mu(T)=0$, Equation \eqref{eq:B cap T conditional expectation} is obviously satisfied. Consequently we have 
\begin{align}\label{eq:B Conditional Expectation}
 \int_B Z \, \textup{d} \mu = \int_B X \, \textup{d} \mu 
\end{align}
for all $B \in  \mathcal{B}_\cap$.
With the help of Lemma \ref{lem:sigma B schnitt = hut B} and Lemma \ref{lem:BedingteErwartungErzeugendeSystem},
we conclude that \eqref{eq:B Conditional Expectation} holds for all $B \in \hat{\mathcal{B}}$. This finishes the proof of $Z \in \mathbb{E}(X|\hat{\mathcal{B}})$.

Next, we note that for $\hat{Z} \in \mathbb{E}(X|\hat{\mathcal{B}})$ we have $\mu(Z \neq \hat{Z}) = 0$ by the almost sure uniqueness of the conditional expectation of $X$ given $\hat{\mathcal{B}}$.

Our next goal is to show that $Z \in  \mathbb{E}(X|\mathcal{C})$. Since $\mathcal{B} \subseteq \mathcal{C}$, it follows that $Z$ is $\mathcal{C}$-measurable. 
Since we have already seen that \eqref{eq:B Conditional Expectation} holds for all $B \in \hat{\mathcal{B}}$ it also holds for all $B \in \mathcal{C}$. This shows $Z \in \mathbb{E}(X|\mathcal{C})$.

Again, we note that for $Z_\mathcal{C} \in \mathbb{E}(X|\mathcal{C})$ we have $\mu(Z \neq Z_\mathcal{C}) = 0$ by the almost sure uniqueness of the conditional expectation of $X$ given ${\mathcal{C}}$.

Our last goal is to establish the inclusion $\mathbb{E}(X|\mathcal{C}) \subseteq \mathbb{E}(X|\hat{\mathcal{B}})$ together with the last identity of \eqref{eq:Z bis auf Nullmengen}. To this end we note that we already know that $ \mathbb{E}(X|\mathcal{C}) \subseteq \mathbb{E}(X|\hat{\mathcal{C}})$ and $\mu(Z_\mathcal{C} \neq \hat{Z}) = 0$ for all $Z_\mathcal{C} \in \mathbb{E}(X|\mathcal{C})$ and $\hat{Z} \in \mathbb{E}(X|\hat{\mathcal{C}})$. Therefore it suffices to show that $\hat{\mathcal{C}}=\hat{\mathcal{B}}$. For the proof of the inclusion $\hat{\mathcal{C}}\subseteq \hat{\mathcal{B}}$. We fix a $C \in \mathcal{C}$ and a $T \in \mathcal{T}$. This gives $C \in \hat{ \mathcal{B}}$ by assumption and $ T \in \hat{\mathcal{B}}$ by the definition of $\hat{\mathcal{B}}$, and therefore $C \cap T \in \hat{\mathcal{B}}$. In other words we have $\mathcal{C}_\cap \subseteq \hat{\mathcal{B}}$ and by Lemma \ref{lem:sigma B schnitt = hut B}, we conclude $\hat{\mathcal{C}} = \sigma(\mathcal{C}_\cap) \subseteq \hat{\mathcal{B}}$. The converse inclusion follows by $\mathcal{B} \subseteq \mathcal{C}$.
\end{proof}

In the following we need the space 
\begin{align*}
\mathcal{L}^2(\mathcal{A}) := \left\{ g: \Omega \to \mathbb{R} \, \middle| \, g \, \, \textup{is measurable}, \, \int_\Omega g^2 \, \textup{d} \mu < \infty \right\}
\end{align*}
equipped with the usual $\Vert \mycdot \Vert_{\mathcal{L}^2(\mathcal{A})}$ semi-norm. Moreover, for $f\in \mathcal{L}^2(\mathcal{A})$ we denote its $\mu$-equivalence class by $[f]$. Note that we have $[f] \in L^2(\mathcal{A})$.

\begin{lemma}\label{lem:G Hut}
Let $G \subseteq \mathcal{L}^2(\mathcal{A})$ be non empty and 
\begin{align}\label{eq:G Hut}
\hat{G}:= \left\{ g \in \mathcal{L}^2(\mathcal{A}) \, \middle| \, \exists (g_n) \subseteq G: [g_n] \to [g] \, \textup{in } L^2(\mathcal{A}) \right\}.
\end{align}
Then for $\mathcal{B}:=\sigma(G)$ we have $\mathcal{B} \subseteq \sigma(\hat{G}) \subseteq \hat{\mathcal{B}}$.
\end{lemma}

\begin{proof}
The statement $\mathcal{B} \subseteq \sigma(\hat{G})$ follows directly from the fact that $G \subseteq \hat{G}$.

Next, we prove the statement $\sigma(\hat{G}) \subseteq \hat{\mathcal{B}}$.
To this end it suffices to show that all $g \in \hat{G}$ are $\hat{\mathcal{B}}$-measurable. 
Let us therefore fix a $g \in \hat{G}$. Then there exists a sequence $(g_n) \subseteq G$ such that $[g_n] \to [g]$ in $L^2(\mathcal{A})$. 
By the properties of convergence in $L^2(\mathcal{A})$, there exists a subsequence $(g_{n_k})$ such that $g_{n_k} \to g$ $\mu$-almost surely. 
Consequently the set $T := \{ \omega \in \Omega \mid g_{n_k}(\omega) \to g(\omega) \}$ is $\mathcal{A}$-measurable with $\mu(T) = 1$, i.e., we have $T \in \mathcal{T}$.

Let us now define $\hat{g}_{n_k} := \chi_T g_{n_k}$ and $\hat{g} := \chi_T g$. Obviously we have $\hat{g}_{n_k} \to \hat{g}$ pointwise. 
Moreover, the definition of $\hat{G}$ ensures $\hat{g}_{n_k}, \hat{g} \in \hat{G}$. Furthermore, $g_{n_k} \in G$ shows that $g_{n_k}$ is $\mathcal{B}$-measurable.  
Since $\chi_T$ is $\mathcal{T}$-measurable and $\mathcal{T} \subseteq\hat{\mathcal{B}}$, it follows that $\hat{g}_{n_k} = \chi_T g_{n_k}$ is $\hat{\mathcal{B}}$-measurable. 
Thus, the limit function $\hat{g}$ is also $\hat{\mathcal{B}}$-measurable.

With these preparations we now show that $g^{-1}(A) \in \hat{\mathcal{B}}$ for all measurable $A \subseteq \mathbb{R}$. To this end we note that 
\begin{align*} 
g^{-1}(A) &= \{ \omega \in \Omega \, | \, g(\omega) \in A \} \\
&=  \bigl(  \{ \omega \in \Omega \, |\, g(\omega) \in A \} \cap T \bigr) \cup \bigl( \{ \omega \in \Omega \, | \, g(\omega) \in A \} \setminus T \bigr) \\
&= \bigl( \{ \omega \in \Omega \, | \, \hat{g}(\omega) \in A \} \cap T \bigr) \cup \bigl( \{ \omega \in \Omega \, | \, g(\omega) \in A \} \setminus T \bigr).
\end{align*}
Since $\{ \omega \in \Omega \mid \hat{g}(\omega) \in A \} \in \hat{\mathcal{B}}$ and $T \in \hat{\mathcal{B}}$, it follows that $\{ \omega \in \Omega \mid \hat{g}(\omega) \in A \} \cap T \in \hat{\mathcal{B}}$. Furthermore, we have
\begin{align*}
\mu\bigl( \{ \omega \in \Omega \, | \, g(\omega) \in A \} \setminus T \bigr) &= \mu \bigl( \{ \omega \in \Omega \, | \, g(\omega) \in A \} \cap (\Omega \setminus T) \bigr) \\
&\leq \mu(\Omega \setminus T) = 0.
\end{align*}
Therefore we have $\{ \omega \in \Omega \mid g(\omega) \in A \} \setminus T \in \hat{\mathcal{B}}$ and by combining this with our previous consideration we conclude $g^{-1}(A) \in \hat{\mathcal{B}}$.
\end{proof}

With these preparations and under \hyperref[Assumption:A]{Assumption A} we now define  
\begin{align*}
\mathbb{E}(X|G_Y):= \mathbb{E}(X|\s(\hat{G})),
\end{align*} 
where 
$G:=\{  f'(Y) \, | \, f' \in F' \}$ and $\hat{G}$ is defined by \eqref{eq:G Hut}. The following theorem describes $\mathbb{E}(X|G_Y)$.

\begin{theorem}\label{thm:BedingteErwartungenHüte}
Let \hyperref[Assumption:A]{Assumption A} be satisfied. Then, we have
\begin{align*}
\mathbb{E}(X|\sigma(Y)) \subseteq \mathbb{E}(X|G_Y) \subseteq \mathbb{E}(X|\widehat{\sigma(Y)}).
\end{align*}
Furthermore, for all $Z \in \mathbb{E}(X|\s(Y))$, $Z_G \in \mathbb{E}(X|G_Y)$, and $\hat{Z} \in \mathbb{E}(X|\widehat{\s(Y)})$, we have
\begin{align*}
\mu(Z \neq Z_G) = \mu(Z \neq \hat{Z}) = \mu(Z_G \neq \hat{Z}) = 0.
\end{align*}
\end{theorem}

\begin{proof}
We set $\mathcal{B}:= \s(G)$ and obtain $\mathcal{B} \subseteq \s(\hat{G}) \subseteq \hat{\mathcal{B}}$ by Lemma \ref{lem:G Hut}. For $\mathcal{C}:=\s(\hat{G})$ Lemma \ref{lem:Z bis auf Nullmengen} thus shows 
\begin{align*}
\mathbb{E}(X|\mathcal{B}) \subseteq \mathbb{E}(X|\sigma(\hat{G})) \subseteq \mathbb{E}(X|\hat{\mathcal{B}})
\end{align*}
and \eqref{eq:Z bis auf Nullmengen}. Finally, Lemma \ref{lem:algebren gleich} shows $\s(Y)=\s(G)=\mathcal{B}$, thus the assertion follows.
\end{proof}

\begin{lemma}\label{lem:Bounded Gaussian random variables}
Let $(r_j) \sim \mathcal{N}(0,1)$ be a sequence of i.i.d. random variables, then 
\begin{align*}
r_j \cdot \left(\sqrt{\ln( j \ln^2(j+2))}\right)^{-1} 
\end{align*}
is a bounded sequence $\mu$-almost everywhere.
\end{lemma}
\begin{proof}

We define 
\begin{align*}
A_j := \left\{ \left| r_j \cdot \left(\sqrt{\ln( j \ln^2(j+2))}  \right)^{-1} \right| > \sqrt{2} \right\}
\end{align*}
and $A := \limsup_{j \to \infty} E_j$. By the Borel-Cantelli lemma, we have $\mu(A)=0$ if $\sum_{j=1}^\infty \mu(A_j) < \infty$. 
Using a well known tail bound for one dimensional standard normal random variables, see \cite[Lemma B.4]{giraud2021introduction} we find $\mu(|r_j|>t) \leq 2 \mathrm{e}^{-t^2/2}$ for all $t>0$. We thus obtain  
\begin{align*}
\sum_{j=1}^\infty \mu(A_j) 
= \sum_{j=1}^\infty \mu\left( |r_j| > \sqrt{2} \cdot \sqrt{\ln( j \ln^2(j+2))}\right) 
&\leq \sum_{j=1}^\infty 2 \mathrm{e}^{-\left(2 \cdot \sqrt{\ln( j \ln^2(j+2))}^2/2 \right) } \\
&= \sum_{j=1}^\infty 2 \mathrm{e}^{-\left(2 \cdot \ln( j \ln^2(j+2))/2 \right) } \\
&= \sum_{j=1}^\infty \frac{2}{ j \ln^2(j+2)  } \\
&\leq \frac{1}{ \ln^2(3)} + 2 \sum_{j=2}^\infty \frac{1}{j \ln^2(j)} \, .
 \end{align*} 
The latter series converges by \cite[Theorem 3.29]{rudin1976principles}, thus $\mu(A)=0$. 
\end{proof}

\begin{lemma}\label{lem:GY=GE}
Let $X:\Omega \to E$ be a Gaussian random variable and $\mathcal{E} \subset E'$. Then there exists a Gaussian random variable $Y:\Omega \to \ell^2(\mathbb{N})$ such that for all $Z_\mathcal{E} \in \mathbb{E}\left(X| \sigma ( \{e'(X) \, : \, e' \in \mathcal{E} \} ) \right)$ and $Z \in \mathbb{E}(X|Y)$ we have $\mu(Z_\mathcal{E}\neq Z)=0$. Moreover we have 
\begin{align*}
G_Y=G_\mathcal{E}:=\overline{\textup{span} \{ e'(X) \, | \, e' \in \mathcal{E} \}}^{\Vert \mycdot \Vert_{L^2(\mathcal{A})}}.
\end{align*}
\end{lemma}

\begin{proof}
For simplicity we prove the statement only for $ \dim   G_\mathcal{E}    = \infty$. The finite dimensional cases are analogous yet simpler. 

Let us fix an ONB $([r_j])$ of $G_\mathcal{E}$ such that $r_j \in \textup{span} ( \{e'(X) \, | \, e' \in \mathcal{E} \} )$ for all $ j \geq 1 $. Here we note that this is possible by the Gram-Schmidt construction if in each step we pick a vector in $\textup{span} ( \{e'(X) \, | \, e' \in \mathcal{E} \} )$ that is linearly independent of the previously constructed members of the ONB.

 We further note that $r_j \sim \mathcal{N}(0,1)$.
Moreover for each $(r_{j_1}, \dots , r_{j_m})$ we know that for linear combinations we have
$\sum_{i=1}^m a_i r_{j_i}\in \textup{span}(\mathcal{E} \circ X )$
 and hence $(r_{j_1}, \dots , r_{j_m})$ is a Gaussian vector. Since its components are pairwise uncorrelated we conclude that $(r_{j_1}, \dots , r_{j_m})$ are independent. In other words, $(r_j)$ is an i.i.d. sequence.

Let $(e_j)$ be the standard ONB of $\ell^2(\mathbb{N})$ and 
\begin{align*}
\alpha_j&:= \left( \sqrt{\ln( j \ln^2(j+2))}\right)^{-1}, \\
Y_n&:= \sum_{j=1}^n \alpha_j r_j   \frac{e_j}{j} \, . 
\end{align*} 
We note that by Lemma \ref{lem:Bounded Gaussian random variables} we $\mu$-almost surely have
\begin{align*}
\left\Vert \left( \alpha_j r_j  \frac{1}{j} \right) \right\Vert_{\ell^2(\mathbb{N})}^2 
\leq \left\Vert \left( \alpha_j r_j   \right) \right\Vert_{\ell^\infty(\mathbb{N})}^2 \cdot  \left\Vert \left( \frac{1}{j} \right) \right\Vert_{\ell^2(\mathbb{N})}^2 < \infty,
\end{align*} 
and thus 
\begin{align*}
\sum_{j=1}^\infty \alpha_j r_j   \frac{e_j}{j} \,
\end{align*} 
converges almost everywhere in $\ell^2(\mathbb{N})$. We denote the set where the series does converge by $T$ and we set 
\begin{align*}
Y(\omega) := \begin{cases} 
\sum_{j=1}^\infty \alpha_j r_j (\omega) \frac{e_j}{j}, \, &{\omega \in T}, \\
0, &{ \text{else}}. 
\end{cases} \, 
\end{align*}
Next we prove that $Y:\Omega \to \ell^2(\mathbb{N})$ is a Gaussian random variable.
To this end let us fix an $a' \in \ell^2(\mathbb{N})$. Then, for $\omega \in T$, we have 
\begin{align*}
a'(Y-Y_n) 
= a' \left( \sum_{j=n+1}^\infty \alpha _j r_j (\omega) \frac{e_j}{j} \right) = \sum_{j=n+1}^\infty \alpha_j \frac{a'(e_j)}{j} r_j (\omega).
\end{align*} 
By Parseval's identity and the fact that $([r_j])$ is an orthonormal system in $L^2(\mathcal{A})$ we conclude
\begin{align}\label{eq:ayinl2}
\int_\Omega \left[ a'(Y-Y_n) \right]^2 \, \textup{d} \mu 
= \sum_{j=n+1}^\infty   \left(  \alpha_j \frac{a'(e_j)}{j} \right)^2 
\leq \Vert a' \Vert_{\ell^2(\mathbb{N})}^2   \sum_{j=n+1}^\infty \frac{\alpha_j^2}{j^2} \to 0. 
\end{align}
Due to the fact that each $a'(Y_n)$ is a Gaussian random variable and $a'(Y_n) \to a'(Y)$ in $L^2(\mathcal{A})$ it follows that $Y$ is indeed a Gaussian random variable. 

Our next goal is to establish the inclusion
\begin{align}\label{eq:EXInclusion}
\sigma(\mathcal{E} \circ X) \subseteq \widehat{\sigma(Y)} \subseteq \widehat{\sigma(\mathcal{E} \circ X)}.
\end{align}
To this end we fix a $g \in \mathcal{E} \circ X$. Then there exists a sequence $(b_j) \in \ell^2(\mathbb{N})$ such that 
$g = \sum_{j=1}^\infty b_j r_j$ in $\mathcal{L}^2(\mathcal{A})$. By 
\begin{align}\label{eq:rjinGy}
e_j(Y)(\omega) = \begin{cases}
\alpha_j r_j (\omega)   \frac{1}{j}, &\text{ if } \omega \in T, \\
0, &\text{ else }
\end{cases}
\end{align}  
we conclude that
\begin{align*}
\chi_T r_j = \frac{j}{\alpha_j} e_j(Y).
\end{align*} 
Therefore, $\chi_T r_j$ is $\s(Y)$-measurable and thus also $\widehat{\s(Y)}$-measurable. Using $r_j=\chi_T r_j + \chi_{\Omega \setminus T} r_j$ we conclude that $r_j$ is $\widehat{\sigma(Y)}$-measurable. Consequently, each finite sum $\sum_{j=1}^m b_j r_j$ is $\widehat{\s(Y)}$-measurable. Since a subsequence of these finite sums converge almost surely to $g$ we conclude that $g$ itself is $\widehat{\sigma(Y)}$-measurable. This shows the first inclusion $\s(\mathcal{E} \circ X) \subseteq \widehat{\s(Y)}$.

For the second inclusion $\widehat{\sigma(Y)} \subseteq \widehat{\s(\mathcal{E}\circ X)}$ we first prove $\s(Y) \subseteq \widehat{\s(\mathcal{E}\circ X)}$. To this end, we note that for $a' \in \ell^2(\mathbb{N})$ we have 
\begin{align}\label{eq:ayreihe}
a'(Y) = \sum_{j=1}^\infty \chi_T \alpha_j r_j  \frac{a'(e_j)}{j}
\end{align}
with pointwise convergence.
We further note that $\chi_T$ is $\mathcal{T}$-measurable and that all $r_j$ are $\sigma(\mathcal{E} \circ X)$-measurable since by our construction they are finite linear combinations of elements in $\mathcal{E} \circ X$. Therefore, $a'(Y)$ is $\widehat{\s(\mathcal{E} \circ X)}$-measurable and by Lemma \ref{lem:algebren gleich} we find $\s(Y) \subseteq \widehat{\s(\mathcal{E}\circ X)}$. By Lemma \ref{lem:sigma B schnitt = hut B} we obtain $\widehat{\sigma(Y)} \subseteq \widehat{\s(\mathcal{E}\circ X)}$. 

Combining \eqref{eq:EXInclusion} with Theorem \ref{thm:BedingteErwartungenHüte} we obtain the first assertion.

For the proof of $G_\mathcal{E}\subseteq G_Y$ we first note that $[r_j] \in G_Y$ by \eqref{eq:rjinGy}. Since $([r_j])$ is an ONB of $G_\mathcal{E}$ and $G_Y$ is a closed subspace of $L^2(\mathcal{A})$ we then find the desired inclusion.

To prove the other inclusion $G_Y \subseteq G_\mathcal{E}$, we fix an $a' \in \ell^2(\mathbb{N})$. Analogously to \eqref{eq:ayinl2} we have 
\begin{align*}
\left( \alpha_j \frac{a'(e_j)}{j} \right) \in \ell^2(\mathbb{N}).
\end{align*}
By \eqref{eq:ayreihe} we conclude that $a'(Y) \in G_\mathcal{E}$, and since $G_\mathcal{E}$ is a closed space we then find $G_Y \subseteq G_\mathcal{E}$.
\end{proof}

\begin{proof}[Proof of Theorem \ref{Theorem: Kernel Conditioning}]
    We set $G_\mathcal{E}:= \overline{\textup{span}\{ e'(X) \, | \, e' \in \mathcal{E} \}}^{\Vert \cdot \Vert_{L^2(\mathcal{A})}}$ and use Lemma \ref{Lemma: e' commutes with conditioning} to obtain
\begin{align*}
e'(X_\mathcal{E}) = e'\left( \mathbb{E}(X| \sigma( \{ e'(X) \, | \, e' \in \mathcal{E} \}) \right) =  \mathbb{E}(e'(X) | \sigma( \{ e'(X) \, | \, e' \in \mathcal{E} \}) )
\end{align*}
for all $e'\in E'$.
We now fix a Gaussian random variable $Y$ according to Lemma \ref{lem:GY=GE}. This leads to 
\begin{align*}
\mathbb{E}(e'(X) | \sigma( \{ e'(X) \, | \, e' \in \mathcal{E} \}) ) = \mathbb{E}(e'(X)|Y).
\end{align*}
By Theorem \ref{the:Theorem3.3Steinwart} and Lemma \ref{lem:GY=GE} we also have 
   \begin{align*}
   \mathbb{E}(e'(X)|Y) = \Pi_{G_Y} e'(X) = \Pi_{G_\mathcal{E}} e'(X)=e'\left( \Pi_{G_\mathcal{E}} X \right),
   \end{align*} 
   where in the last step we used Lemma \ref{lem:Projektions Darstellungen}. In summary we have 
   \begin{align}\label{e'Xe=PiGX}
   e'(X_\mathcal{E}) = e'(\Pi_{G_\mathcal{E}} X)
   \end{align}
   almost surely.
   We conclude by Lemma \ref{lem:Hahn-Banach für Zufallsvariablen}, that $\Pi_{G_\mathcal{E}}X=X_\mathcal{E}$ and by Lemma \ref{lem:Projektions Darstellungen}, that $X_\mathcal{E}$ is a Gaussian random variable.

Next we prove that the kernel $k_{\mathcal{E}}$ is given by 
    \begin{align*}
    k_{\mathcal{E}} (\mycdot,e') =  \Pi_{H(\mathcal{E})} k_X(\mycdot,e')\, \quad \, \textup{with} \,  \, e' \in B_{E'} \, .
    \end{align*} 
Recall that the operator $U_X$ from \eqref{eq: U_X} is an isometry and for $e' \in \mathcal{E}$ we have 
\begin{align}\label{eq:UXaufKe}
U_X k_X( \mycdot , e') = V_X^{-1} \left( \int_\Omega e'(X) X \, \textup{d}\mu \right) =e'(X).
\end{align}
By taking the $\textup{span}$ and the closure we conclude $U_X{H(\mathcal{E})} = {G_\mathcal{E}}$. 
Using the reproducing property in $H_X$, Equation \eqref{eq:UXaufKe}, Lemma \ref{lem:Orthogonal projections are similiar}, and Lemma \ref{lem:Projektions Darstellungen} in combination with Equation \eqref{e'Xe=PiGX} we find  
    \begin{align*}
\left(   \Pi_{H(\mathcal{E})}k_X(\mycdot, e_1') \right) (e_2') 
&=  \langle \Pi_{H(\mathcal{E})} k_{X}(\mycdot, e_1') ,  k_{X} (\mycdot,e_2') \rangle_{H_X} \\
&= \langle U_X \Pi_{H(\mathcal{E})} U_X^* e_1'(X) , U_X \Pi_{H(\mathcal{E})} U_X^* e_2'(X) \rangle_{G_X} \\ 
&= \langle  \Pi_{G_{\mathcal{E}}} e'_1(X), \Pi_{G_{\mathcal{E}}} e'_2(X) \rangle_{G_X} \\
&= \int_\Omega e'_1(X_\mathcal{E}) e_2' ( X_\mathcal{E}) \, \textup{d} \mu \\
&=  \langle \textup{cov}(X_\mathcal{E}) e'_1, e_2' \rangle_{E,E'}=k_\mathcal{E}(e_1',e_2') \, 
    \end{align*}
for all $e'_1,e_2' \in B_{E'}$.

Next we note that $X-X_\mathcal{E}$ is Gaussian since $X,X_\mathcal{E}$ are jointly Gaussian by Lemma \ref{lem:Projektions Darstellungen}.
Moreover, for $e_1',e_2' \in B_{E'}$, we obtain by Lemma \ref{lem:Covariance Rechenregel}
\begin{align*}
k_{X-X_\mathcal{E}}(e_1',e_2') &= \langle \textup{cov}(X-X_{\mathcal{E}}) e_1',e_2' \rangle_{E,E'}  \\
&= \langle \textup{cov}(X)e_1',e_2' \rangle_{E,E'} - \langle  \textup{cov}(X_\mathcal{E}) e_1',e_2' \rangle_{E,E'} \, .
\end{align*}

       Finally we establish the last assertion. With what we have already proven we conclude
    \begin{align*}
        k_{X-X_\mathcal{E}}(e',e') = \Vert k_{X} (\cdot, e') - \Pi_{H(\mathcal{E})} k_{X} (\cdot,e') \Vert_{H_X}^2 \qquad \text{for } \, e' \in B_{E'} \, .
    \end{align*} 
With the operator norm definition this leads to
      \begin{align*}
       \Vert \textup{cov}(X-X_\mathcal{E}) \Vert_{E'\to E} 
 =         \sup_{e' \in B_{E'}}  \Vert k_{X} (\cdot, e') - k_\mathcal{E} (\cdot,e') \Vert_{H_X}^2.
      \end{align*} 
      The second identity follows  by Lemma \ref{lem:Covariance Rechenregel}.
\end{proof}

\begin{proof}[Proof of Corollary \ref{Coroallry:No Noise}]
We apply Theorem \ref{Theorem: Kernel Conditioning} with the set 
\begin{align*}
\mathcal{E}:= \{ 0 \} \cup \left\lbrace \frac{L' f'}{\Vert L' f' \Vert_{E'}} \, \middle|  \, f' \in F' \text{ with }  L'f' \neq 0 \right\},
\end{align*}
where we note that $Z=X_\mathcal{E}$ since $(L'f')(X)=f'(LX)=f'(Y)$ for all $f' \in F'$.
Therefore it remains to prove that $G_Z=G_Y$ holds true. 
To this end, we note that 
\begin{align*}
G_Y= \overline{\textup{span} \{ f'(LX) \, | \, f'\in F'\} }^{\Vert \mycdot \Vert_{L^2(\mathcal{A})}} = \overline{\textup{span} \{ L' f' \circ X \, | \, f'\in F'\} }^{\Vert \mycdot \Vert_{L^2(\mathcal{A})}} \subseteq G_X
\end{align*}
holds true.
Now Lemma \ref{lem:Subsets} implies $G_Y=G_Z$.
\end{proof}

\begin{proof}[Proof of Lemma \ref{Lemma:M_W Bounded}]
By the B.L.T. Theorem, see \cite[Theorem 1.7]{reed1972methods} there exists a bounded linear extension $M:F\to E$ of $M_W:W_Y \to W_Z$. Repeating \eqref{eq:M_Wcheating} we then obtain 
\begin{align*}
M Y= M \sum_{j \in J} r_j \int_\Omega r_j Y \, \textup{d}\mu = \sum_{j \in J} r_j M_W \left( \int_\Omega r_j Y \, \textup{d} \mu \right) = \sum_{j \in J} r_j  \int_\Omega r_j Z \, \textup{d} \mu = Z.
\end{align*}  
\end{proof}  
\begin{lemma}\label{lem:isometry+surjective}
 Let \hyperref[Assumption:A]{Assumption A} be satisfied and $W_Y$ be dense in $F$. Then the following statements are equivalent:
  \begin{enumerate}[label={\roman*)},ref=\roman*)]
      \item\label{enum:1} The operator $M_W:W_Y \to W_Z$ is an isometric isomorphism. 
      \item\label{enum:2} The equality $G_Y=G_Z$ holds true. 
  \end{enumerate}
  \end{lemma}
  Note that \ref{enum:1} and \ref{enum:2} are equivalent even if $W_Y$ is not dense in $F$.
\begin{proof}
\atob{\ref{enum:2}}{\ref{enum:1}}
 As discussed around \eqref{eq:M W Definition}, the operator $M_W$ is surjective. 
Furthermore, for $w \in W_Y$ we find a $g \in G_Y=G_Z$ such that $w=V_Y g$. This gives
\begin{align*}
\Vert M_W w \Vert_{W_Z}= \Vert \hat{V}_Z \hat{V}_Y^* V_Y g \Vert_{W_Z} = \Vert \hat{V}_Z g \Vert_{W_Z}= \Vert g \Vert_{G_Z} = \Vert g \Vert_{G_Y} = \Vert V_Y g \Vert_{W_Y}= \Vert w \Vert_{W_Y}.
\end{align*}
 Therefore, $M_W:W_Y \to W_Z$ is an isometric isomorphism.

\atob{\ref{enum:1}}{\ref{enum:2}}
By Lemma \ref{lem:Subsets} we know that $G_Z \subseteq G_Y$. 
To prove the converse inclusion we consider the orthogonal complement $G_Z^\perp$ of $G_Z$ in $G_Y$. Let $g \in G_Z^\perp$. Then by the definition of $M_W$ and Lemma \ref{lem:Unitar} we have
\begin{align*}
 M_W \left(\int_\Omega g Y \, \textup{d} \mu \right) = \hat{V}_Z g = 0.
\end{align*}
By assumption the latter implies $V_Y g =0$ and since $V_Y$ is injective we find $g=0$. Therefore we have $G_Z^\perp = \{0\}$.
\end{proof}

\begin{proof}[Proof of Theorem \ref{Theorem:Extension of M_W}]
We first note that by Lemma \ref{lem:isometry+surjective}, we have that $M_W$ is an isometry. 
    We now define a norm on $W_Z$ by
    \begin{align*}
        \Vert w \Vert_{\Tilde{E}_0} := \Vert M_W^{-1}w \Vert_{F}
    \end{align*}
    for all $w \in W_Z$.
    Since $M_W$ is invertible as an isometric isomorphism, this norm is well-defined. We define $\tilde{E}$ to be the completion of the space $(W_Z,\Vert \mycdot \Vert_{\Tilde{E}_0})$.  
    
    Our first goal is to show that $\Tilde{E}$ is separable. To this end we note that $W_Y \subseteq F$ is separable with respect to the $\Vert \mycdot \Vert_F$-norm since $F$ is separable. Consequently there exists a countable and $\Vert \mycdot \Vert_F$-dense $D \subseteq W_Y$. Clearly it suffices to show that $\hat{D}:=M_W D \subseteq W_Z$ is $\Vert \mycdot \Vert_{\Tilde{E}_0}$-dense in $W_Z$. To this end we fix a $w \in W_Z$. 
    Then there exists a $v \in W_Y$ such that $M_W v =w $, and additionally  a sequence $(v_n) \subset D$ such that $v_n \to v$. For $w_n:=M_W v_n \in \hat{D}$, we then find 
 \begin{align*}
 \Vert w_n - w \Vert_{\Tilde{E}_0}= \Vert M_W^{-1} (w_n - w ) \Vert_F = \Vert v_n - v \Vert_F \to 0.
\end{align*}  
Let us now consider the operator $\tilde{M}_W : W_Y \to \tilde{E}$ that is given by 
\begin{align*}
\Tilde{M}_W w = M_W w. 
\end{align*}
Our next goal is to show that this operator is continuous with respect to the norms $\Vert \mycdot \Vert_F$ and $\Vert \mycdot \Vert_{\tilde{E}_0}$.
To this end we fix $w \in W_Y$, then we have $\tilde{M}_W w \in W_Z$ and thus we find 
\begin{align*}
\Vert \tilde{M}_W w \Vert_{\tilde{E}_0} = \Vert M_W^{-1} \tilde{M}_W w \Vert_F = \Vert w \Vert_F.
\end{align*}
By \cite[Theorem 1.9.1]{megginson2012introduction} there exists a unique bounded and linear extension $\hat{M}:F \to \Tilde{E}$ of $\Tilde{M}_W$. Clearly it is even an isometry.

        Lastly, we show that
        $\hat{M}Y =Z$
    holds true. To this end, let $(r_j)_{j \in J} \subset G_Y$ be an ONB. By Theorem \ref{Theorem:Conditioning orthogonal projection} we then find 
        \begin{align*}
        Y&=\mathbb{E}(Y|Y) = \sum_{j \in J} r_j \int_\Omega r_j Y \, \textup{d} \mu,
        \end{align*}
        where the convergence is pointwise almost sure in $F$.
Applying $\hat{M}$ leads to
        \begin{align*}
      \hat{M} Y
      =  \hat{M} \left( \sum_{j \in J} r_j \int_\Omega r_j Y \, \textup{d} \mu \right) 
      = \sum_{j \in J} r_j \hat{M} \left( \int_\Omega r_j Y \, \textup{d} \mu \right) 
      &= \sum_{j \in J} r_j M_W \left( \int_\Omega r_j Y \, \textup{d} \mu \right) \\ 
      &= \sum_{j \in J} r_j \int_\Omega r_j Z \, \textup{d} \mu \\
	  &= \mathbb{E}(Z|Y)       \\
      &=Z,
    \end{align*}
    where in the fourth step we used the definition of $M_W$.
\end{proof}

\begin{proof}[Proof of Theorem \ref{Theorem:conditional_variance}]        
The first assertion can be found for example in \cite[Theorem 3.3 vi).]{steinwart2024conditioning}, but here we give an independent proof.
        To this end we define $X_u:=X-Z$, this gives 
    \begin{align*}
        Z=\mathbb{E}(X|Y)= \mathbb{E}(X_u+Z|Y)= \mathbb{E}(X_u|Y)+\mathbb{E}(Z|Y).
    \end{align*}
    Using the definition of the conditional expectation we conclude $\mathbb{E}(Z|Y)=Z$, which implies $\mathbb{E}(X_u|Y)=0$.
 We conclude that 
    \begin{align*}
        \langle g_u,g_y \rangle_{L^2(\mathcal{A})} = 0 
    \end{align*}
    for all $g_u \in G_{X_u}$ and $g_y \in G_Y$. Additionally, since $\mathbb{E}(X)=\mathbb{E}(Z)=0$
    we have 
    \begin{align*}
        \mathbb{E}(X_u)=0.
    \end{align*} 
Utilizing Lemma \ref{lem:Covariance Rechenregel} leads to
\begin{align*}
    \langle \textup{cov}(X|Y) e',e' \rangle_{E,E'} =     \langle \textup{cov}(X_u+Z|Y) e',e' \rangle_{E,E'} 
     &= \mathbb{E}( \left(e'(X_u+Z- \mathbb{E}(X_u+Z|Y))\right)^2 |Y)  \\
     &= \mathbb{E}( \left(e'(X_u+Z- Z)\right)^2 |Y) \\
     &= \mathbb{E}( \left(e'(X_u)\right)^2 | Y) \\
     &=  \langle \textup{cov}(X_u) e',e' \rangle_{E,E'} \\
     &=  \langle \textup{cov}(X-Z) e',e' \rangle_{E,E'} \\
     &= \langle \textup{cov}(X) e',e' \rangle_{E,E'}- \langle \textup{cov}(Z) e',e' \rangle_{E,E'} \, ,
\end{align*}
here in the fifth step we used that $e'(X_u)$ is independent of $Y$, since $\mathbb{E}(e'(X_u)|Y)=0$.
Moreover, we used the well known fact that it suffices to calculate the diagonal of the covariance operator, see Lemma \ref{lem:LemmaB2Steinwart}. 

For the second assertion we use $MY=Z$ and some well known formulas for cross covariance operators, see Lemma \ref{lem:LemmaB2Steinwart}: 
\begin{align*}
 \textup{cov}(X|Y) 
    = \textup{cov}(X)-  \textup{cov}(Z)
        = \textup{cov}(X)-  M \textup{cov}(Y) M' \, . 
\end{align*} 
\end{proof} 

\begin{proof}[Proof of Theorem \ref{Theorem:main convergence rate generalized}]
We set 
\begin{align*}
   Z_n&{:=}\mathbb{E}(X| \sigma \{ f'(Y) \, : \, f' \in F_n' \} ) \\
   Z&{:=}\mathbb{E}(X|Y).
\end{align*}
By applying Theorem \ref{Theorem:conditional_variance} and Lemma \ref{lem:Covariance Rechenregel} we obtain
\begin{align*}
\textup{cov}(X|Y_n)-\textup{cov}(X|Y)=\left( \textup{cov}(X)-\textup{cov}(Z_n) \right) - \left( \textup{cov}(X)-\textup{cov}(Z) \right)
&=\textup{cov}(Z)-\textup{cov}(Z_n) \\
&=\textup{cov}(Z-Z_n).
\end{align*} 
    Using \hyperref[Assumption:M]{Assumption M}, Theorem \ref{Theorem:Conditioning orthogonal projection} and given an ONB $(r_j)_{j \in J} \subseteq G_{Y_n}$ we obtain
\begin{align*}
M Y_n= M \sum_{j \in J} r_j \int_\Omega r_j Y \, \textup{d} \mu = \sum_{j \in J} r_j M \int_\Omega r_j Y \, \textup{d} \mu = \sum_{j \in J} r_j \int_\Omega r_j Z \, \textup{d} \mu = Z_n.
\end{align*}
Now we prove the inequality via
    \begin{align*}
        \Vert   \textup{cov}(X|Y) - \textup{cov}(X|Y_n) \Vert_{E' \to E} &= \Vert \textup{cov}(Z) - \textup{cov}(Z_n) \Vert_{E' \to E} \\
        &= \Vert \textup{cov}(MY) - \textup{cov}(MY_n) \Vert_{E' \to E} \\
        &\leq \Vert M \Vert_{F \to E}^2 \Vert \textup{cov}(Y) - \textup{cov}(Y_n) \Vert_{F' \to F}.
    \end{align*}
\end{proof}

\begin{lemma}\label{lem:compact F}
Given a centered Gaussian random variable $X:\Omega \to E$ the set 
\begin{align*}
\mathcal{F}:=\{ k_X(\mycdot , e') \, \vert \, e' \in B_{E'} \}
\end{align*}
is compact in $H_X$.
\end{lemma}
\begin{proof}
Note, that by \cite[Theorem 8.2.6.]{stroock2010probability} the map $T:E' \to W_X$ with $Te'=\int_\Omega e'(X) X \, \textup{d} \mu$ is continuous from the weak-$*$ topology into the norm topology on $W_X$. Because $\iota_X:W_X \to H_X$ is continuous, this implies that $\iota_X T:E' \to H_X$ is continuous from the weak-$*$ topology into the norm topology on $H_X$. Moreover, we have that $\iota_X T e' = k_X(\mycdot , e')$.
By Lemma \ref{lem:countable pointwise dense subset} we have that $B_{E'}$ is compact in the weak-$*$-topology and since $\iota_X TB_{E'}=\mathcal{F}$ we have that $\mathcal{F}$ is compact in $H_X$.
\end{proof}

\begin{proof}[Proof of Corollary \ref{Corollary:Polynomial Convergence Rate P-Greedy}]
We apply Theorem \ref{Theorem:main convergence rate generalized}, meaning we only have to show that 
\begin{align*}
\Vert \textup{cov}(Y) - \textup{cov}(Y_n) \Vert_{F' \to F} \leq 2 \min_{1 \leq m < n } \gamma^{-2} d_m^{\frac{2(n-m)}{n}}(\mathcal{F}) 
\end{align*}
holds true. By Lemma \ref{lem:compact F} we have that $\mathcal{F}$ is compact and by using Corollary \ref{cor:convergence rate P-greedy}, we obtain 
\begin{align*}
\sup_{ f \in B_{H_Y}} \Vert f - \Pi_{F_n'} f \Vert_{C(B_{F'})} \leq \sqrt{2} \min_{1 \leq m < n } \gamma^{-1} d_m^{\frac{(n-m)}{n}}(\mathcal{F}),
\end{align*}
with $\Pi_{F_n'}$ being the orthogonal projection in $H_Y$ onto the subspace $\textup{span} \{ k_Y(\mycdot, f_n') \, | \, f_n' \in F_n'\}$. Utilizing Theorem \ref{Theorem: Kernel Conditioning} and \cite[Lemma 2.3]{santin2017convergence} we end up with 
\begin{align*}
\sup_{ f \in B_{H_Y}} \Vert f - \Pi_{F_n'} f \Vert_{C(B_{F'})}^2 &= \sup_{ f \in B_{H_Y}} \sup_{f' \in B_{F'}} | (f - \Pi_{F_n'} f) (f')|^2 \\
&= \sup_{f' \in B_{F'}} \Vert k_Y(\mycdot,f') - \Pi_{F_n'} k_Y(\mycdot, f') \Vert_{H_Y}^2 \\
&= \Vert \textup{cov}(Y)-\textup{cov}(Y_n) \Vert_{F' \to F}.
\end{align*}
Again using Theorem \ref{Theorem: Kernel Conditioning} and \cite[Lemma 2.3]{santin2017convergence} we obtain 
\begin{align*}
d_n(\mathcal{F})^2 &\leq \sup_{f \in B_{H_Y}} \Vert f - \Pi_{\mathfrak{F}_n'} f \Vert_{C(B_{F'})}^2 \\
&=  \Vert \textup{cov}(Y)-\textup{cov}(Y_n^\star) \Vert_{F' \to F} \leq Cn^{-\alpha}. 
\end{align*}
Using \cite[Corollary 3.3 (ii)]{devore2013greedy}, we end up with 
\begin{align*}
\Vert \textup{cov}(Y)-\textup{cov}(Y_n) \Vert_{F' \to F} \leq 2^{5\alpha+1}\gamma^{-2} C n^{-\alpha}.
\end{align*}
\end{proof}
\begin{lemma}\label{lem:minimumkernel}
Let $a,b,c \in \mathbb{R}$ such that $a<b$ and $c^2+a > 0$. Moreover let $k:[a,b] \times [a,b] \to \mathbb{R}$ be the function given by  
\begin{align*}
k(t,s):=c^2 + \min(t,s).
\end{align*}
Then $k$ is a kernel and the scalar product of the associated RKHS $H$ is given by 
\begin{align}\label{eq:scalarproductH1}
\langle f ,g \rangle_H=\frac{f(a)g(a)}{c^2+a} + \int_a^b f'(t) g'(t) \, \textup{d} t
\end{align}
for all $f,g\in H$.
\end{lemma}
\begin{proof}
Clearly $k$ is a kernel.
Let us consider the set $H=H^1([a,b])$, equipped with the scalar product from \eqref{eq:scalarproductH1}. It is a simple  routine to verify that $(H,\Vert \mycdot \Vert_H)$ is an RKHS and that $k(s,\mycdot)\in H$ for all $s\in [a,b]$. Consequently, it suffices to show that $k$ is the reproducing kernel of $H$. To this end, 
 we pick an $f \in H$. For $s \in [a,b]$, we then have 
\begin{align*}
\langle f , c^2+ \min(s,\mycdot) \rangle_H
= \frac{f(a)(c^2 +a)}{c^2+a} + \int_a^s f'(t) \, \textup{d}t 
= f(a) + f(s)-f(a) 
=f(s).
\end{align*}
\end{proof}

\section{Additional Examples}\label{sec:6}
Lastly we give some final example on how to deal with a general operator $L:E  \to F$ and some noisy observation such that $Y=LX+N$ with $N$ being a Gaussian random variable independent of $X$. 
\begin{example}\label{example:1}
We assume that $W_Y$ is dense in $F$ and that we have a continuous invertible operator $L:E \to F$ such that $Y=LX$. A simple calculation shows  
\begin{align*}
Z= \mathbb{E}(X|Y)= \mathbb{E}(X|LX) = \mathbb{E}(X|X)=X=L^{-1}Y,
\end{align*}
and this suggests $M=L^{-1}$. Let us now verify this.
We first observe that since $L$ is invertible the equality 
\begin{align*}
G_X=G_Y
\end{align*}
holds true.
For $g \in G_X$ the norm of $w :=V_X g \in W_X$ is given by 
\begin{align*}
\Vert L_W w \Vert_{W_Y}=\left\Vert \int_\Omega g Y \, \textup{d} \mu \right\Vert_{W_Y}= \left\Vert g \right\Vert_{G_Y} = \left\Vert g \right\Vert_{G_X}= \left\Vert \int_\Omega g X \, \textup{d} \mu \right\Vert_{W_X}= \Vert w \Vert_{W_X},
\end{align*}
where we used the definition $L_W = \hat{V}_Y \hat{V}_X^*$, see \eqref{eq:Definition of L_W as commutative diagramm}.
Consequently $L_W$ is isometric, and since it is also surjective it is an isometric isomorphism.
Here we used that $L_W^{-1}$ exists because $L$ is invertible.
By the definition of $M_W$ we conclude that
$M_W=L_W^*=L_W^{-1}$.
Since we further know that $L|_{W_X}=L_W$, see Lemma \ref{lem:L_W = L restricted}, we obtain $M_W=L_W^{-1}=L^{-1}|_{W_Y}$. Consequently $L^{-1}$ is indeed a continuous extension of $M_W$ and since $W_Y$ is dense in $F$ it is the only one, in other words we have 
\begin{align*}
M=L^{-1}.
\end{align*} 
\end{example}

The next example generalizes the previous Example \ref{example:1} to $L$ that are not invertible.

\begin{example}\label{example:2}
Let $L:E \to F$ be a bounded operator and $Y=LX$. Then we have $M_W=L_W^\dagger$, where
\begin{align}\label{eq:MinimizationProblem}
L_W^\dagger w:= \textup{argmin} \left\{ \Vert w_x \Vert_{W_X} \, \middle| \, L_W w_x= w \right\}, \qquad \qquad w \in W_Y
\end{align}
is the Moore-Penrose inverse of $L_W$, see \cite{ding1994perturbation}.
To verify this we first note that $L_W^\dagger$ does exist since $\textup{ran}(L_W)=W_Y$ is obviously closed in $W_Y$.
Moreover , we find $G_Y \subseteq G_X$ by $Y=LX$, and obviously $G_Y$ is closed in $G_X$.

We define the space 
\begin{align*}
W_X(Y):=\left\{ \int_\Omega g X \, \textup{d} \mu \,  \middle\vert \, g \in G_Y \right\}.
\end{align*}
Note that the space $W_X(Y)$ is a closed subspace of $W_X$ as $G_Y\subseteq G_X$ is a closed subspace and $W_X(Y)=V_X G_Y$. 
Now let $w \in W_Y$. Then there exists a $g \in G_Y$ such that $V_Yg=w$. We calculate 
\begin{align*}
\Vert w \Vert_{W_Y}
= \left\Vert \int_\Omega g Y \, \textup{d} \mu \right\Vert_{W_Y} 
= \left\Vert \int_\Omega \Pi_{G_Y} g Y \, \textup{d} \mu \right\Vert_{W_Y}  
= \Vert \Pi_{G_Y} g \Vert_{G_Y} 
&= \Vert \Pi_{G_Y} g \Vert_{G_X} \\
&= \left\Vert \int_\Omega \Pi_{G_Y}  g X \, \textup{d} \mu \right\Vert_{W_X},
\end{align*} 
where in the second step we used $g \in G_Y$ and thus $\Pi_{G_Y} g = g $. 
Moreover, since for $U:=V_X^*$ we have $\Pi_{G_Y} U=U \Pi_{W_X(Y)} $, see Lemma \ref{lem:Orthogonal projections are similiar}, we find
\begin{align*}
\left\Vert \int_\Omega \Pi_{G_Y}  g X \, \textup{d} \mu \right\Vert_{W_X}
&=
 \left\Vert \Pi_{W_X(Y)}\left( \int_\Omega   g X \, \textup{d} \mu \right) \right\Vert_{W_X} \\
&= \min \left\{ \Vert \tilde{w} \Vert_{W_X} \, \middle\vert \, \tilde{w}-\int_\Omega g X \, \textup{d} \mu \in W_X(Y)^\perp,\, {\tilde{w} \in W_X} \right\}  .
\end{align*}
Now observe that $G_Y \subseteq G_X$ implies $L_W=\hat{V}_Y \hat{V}_X^*=V_Y V_X^*$.
In addition, for $\tilde{w} \in W_X$ we have 
\begin{align*}
\tilde{w}-\int_\Omega g X \, \textup{d} \mu \in W_X(Y)^\perp \,
\Leftrightarrow \, V_X^* \left( \tilde{w}-\int_\Omega g X \, \textup{d} \mu \right) \in G_Y^\perp \,  
&\Leftrightarrow \quad V_Y V_X^* \left( \tilde{w}-
 \int_\Omega g X \, \textup{d} \mu \right) = 0 \\
&\Leftrightarrow \quad L_W \tilde{w} = \int_\Omega g LX \, \textup{d} \mu.
\end{align*} 
In summary we thus find 
\begin{align*}
\Vert w \Vert_{W_Y} 
= \left\Vert \int_\Omega \Pi_{G_Y}  g X \, \textup{d} \mu \right\Vert_{W_X} 
&= \min \left\{ \Vert \tilde{w} \Vert_{W_X} \, \middle\vert \, \tilde{w}-\int_\Omega g X \, \textup{d} \mu \in W_X(Y)^\perp,\, {\tilde{w} \in W_X} \right\}  \\
 &=\min \left\{ \Vert \tilde{w} \Vert_{W_X} \, \middle\vert \, L_W \tilde{w}    =\int_\Omega g LX \, \textup{d} \mu,\, {\tilde{w} \in W_X} \right\} \\
 &= \left\Vert L_W^\dagger \left( \int_\Omega g Y \, \textup{d} \mu \right) \right\Vert_{W_X} \\
 &= \Vert L_W^\dagger w \Vert_{W_X}.
\end{align*}
Now calculating the adjoint given $ u \in W_X$ and $ v \in W_Y$ we obtain
\begin{align*}
\langle L_W u, v \rangle_{W_Y} = \langle L_W^\dagger L_W u, L_W^\dagger v \rangle_{W_X} = \langle u, (L_W^\dagger L_W )^* L_W^\dagger v \rangle_{W_X} = \langle u, L_W^\dagger v \rangle_{W_X},
\end{align*}
where in the last step we used one of the defining properties of the Moore-Penrose inverse. In other words we have $M_W=L_W^\dagger$.
Whether $M_W$ admits such an extension depends on the setting, compare Example \ref{example:4} and Example \ref{example:Gegenbeispiel}. 
\end{example}

The next example extends the previous example to noisy observations of $LX$.
\begin{example}\label{example:3}
Let $L:E \to F$ be a bounded operator and $N: \Omega \to F$ be a Gaussian random variable independent of $X$. For
\begin{align*}
Y:= LX + N ,
\end{align*}
we then have 
\begin{align}\label{eq:MWDaggerEinbettung}
M_W= \hat{L}_W^\dagger \textup{Id}_L^* 
\end{align}
 with $\hat{L}_W^\dagger $ as in \eqref{eq:MinimizationProblem} and $\textup{Id}_L^*$ being the adjoint of the embedding $\textup{Id}_L: W_{LX} \to W_Y$, where we note that this embedding is well-defined since the independence of $N$ and $X$ implies $G_Y= G_{LX} \oplus G_N$.
To verify \eqref{eq:MWDaggerEinbettung}
we note that for $g \in G_X$
\begin{align*}
L_W \left( \int_\Omega g X \, \textup{d} \mu \right) 
=\int_\Omega g Y \, \textup{d} \mu = \int_\Omega g (LX + N) \, \textup{d} \mu = \int_\Omega g LX \, \textup{d} \mu,  
\end{align*}
where we used that $N$ and $X$ are independent, thus $\int_\Omega g N \, \textup{d} \mu =0$. 
 In other words, we have $L_W=\textup{Id}_L \hat{L}_W$ with 
\begin{align*}
\hat{L}_W: W_X &\to W_{LX}  \\
\int_\Omega g X \, \textup{d} \mu &\mapsto \int_\Omega g LX \, \textup{d} \mu.
\end{align*}
We already showed in Example \ref{example:2} that the adjoint of $\hat{L}_W$ is $\hat{L}_W^\dagger$. In summary, the adjoint of $L_W$ is
\begin{align*}
L_W^*=(\textup{Id}_L \hat{L}_W)^*=\hat{L}_W^* \textup{Id}_L^* = \hat{L}_W^\dagger \textup{Id}_L^*. 
\end{align*}
\end{example}

Example \ref{example:3} shows that one can factorize the conditioning problem into solving the minimization problem for $\hat{L}_W^\dagger$ in \eqref{eq:MinimizationProblem} and calculating the adjoint of an embedding. 
We further note that for invertible $L$, we have $\hat{L}_W^\dagger=L^{-1}|_{W_{LX}}$, which makes the computation of $M_W=L_W^*$ in Example \ref{example:3} easier.

The next example investigates a variant of the noise model in Example \ref{example:3}.

\begin{example}\label{example:3.5}
Let $L:E \to F$ be a bounded operator, $N: \Omega \to E$ be a Gaussian random variable independent of $X$, and
\begin{align*}
Y:= L(X + N)\, .
\end{align*}
Clearly this is a special case of Example \ref{example:3} and in the following we provide an alternative way to compute $M_W$. 

To this end, we define the following operators 
\begin{align*}
&{\hat{L}_W : W_X \to W_{LX}}, &{\tilde{L}_W : W_{X+N} \to W_Y}, \\
&{\qquad \quad \, \, w \mapsto Lw} &{\qquad \quad \, \, w \mapsto L w} \, \, \\
&{\textup{Id} : W_X \to W_{X+N}}, &{\textup{Id}_L : W_{LX} \to W_Y 
}. \\
&{\qquad \, \, \, \,  w \mapsto w} &{w \mapsto w }  \, \, \, \, \, \, \,
\end{align*} 
Repeating the calculations of Example \ref{example:3} we obtain the following commutative diagram 
\begin{equation*}
\begin{tikzcd}
W_X \arrow[rr, "\hat{L}_W"] \arrow[dd, "\textup{Id}"] \arrow[rrdd, "L_W"] &  & W_{LX} \, \arrow[dd, "\textup{Id}_L"] \\
                                                                          &  &                                  \\
W_{X+N} \arrow[rr, "\tilde{L}_W"]                                           &  & W_Y .                            
\end{tikzcd}
\end{equation*}
In other words we have $L_W = \textup{Id}_L \hat{L}_W = \tilde{L}_W \textup{Id}$.
Calculating the adjoint as in Example \ref{example:3} we obtain 
\begin{align*}
M_W = \hat{L}_W^\dagger \textup{Id}_L^* = \textup{Id}^* \tilde{L}_W^\dagger.
\end{align*}
We note that the Moore-Penrose inverses of $\hat{L}_W$ and $\tilde{L}_W$ can differ.
\end{example}

 In our
 final example, we investigate Hilbert-space-valued Gaussian random variables 
 for which our observational $Y$ is based upon a subset of the eigenvectors of the covariance operator of $X$. 

\begin{example}\label{example:6}
   Let $H$ be a separable Hilbert space and $X:\Omega \to H$ be a Gaussian random variable
such that $W_X$ is dense in $H$.
   Then the covariance operator can be viewed as a symmetric and positive operator 
   $\textup{cov}(X): H \to H$. Clearly, $\textup{cov}(X)$ is also compact and our denseness
   assumption ensures that it is injective. 
    Consequently, all its, at most countably many, 
eigenvalues $(\lambda_i)_{i\in I}$ are greater than zero. The corresponding eigenvectors, denoted by $(e_i)_{i\in I}$, 
  form an ONB of $H$ and one can show that the sequence $(\sqrt{\lambda_i} e_i)_{i \in I}$ is an ONB of $W_X$ with 
   \begin{align}\label{eq:Minimierungsproblem}
 \Vert w \Vert_{W_X}^2 = \sum_{i \in I }  \langle w, \sqrt{\lambda_i} e_i \rangle_{W_X}^2 = \sum_{i \in I }  \frac{\langle w, \textup{cov}(X) e_i \rangle_{W_X}^2}{\lambda_i}  =\sum_{i \in I }  \frac{\langle w, e_i \rangle_H^2}{\lambda_i} 
   \end{align}
   for all $w\in W_X$.
     Now,  given a non-empty $J \subseteq I$ we consider the map $L:H \to \ell^2(J)$ given by
      \begin{align*}
          L f := \left( \langle f,e_j\rangle_H \right)_{j \in J}\, . 
      \end{align*}
      Lastly, we set $Y:=L X$. By Example \ref{example:2} we know that $M_W :W_Y\to W_X$ is given by 
\begin{align*}
M_W w = \textup{argmin} \left\{ \Vert w_x \Vert_{W_X} \, \middle| \, L w_x= w \right\}
\end{align*}
for all $w \in W_Y$.  To solve this optimization problem, we fix a 
$w:=(w_j)_{j \in J} \in W_Y \subseteq \ell^2(J)$.
For $w_x\in W_X$
with $Lw_x = w$ we then know $w_x= \sum_{i \in I} \alpha_i \sqrt{\lambda_i} e_i$
for some
  $(\alpha_i) \in \ell^2(I)$, and therefore  $Lw_x = w$ 
implies
\begin{align*}
  ( \alpha_j \sqrt{\lambda_j} )_{j \in J} =   L_W w_x =  (w_j)_{j \in J}. 
\end{align*}
In other words, we have $\a_j = \lb_j^{-1/2} w_j$ for all $j\in J$.
In view of \eqref{eq:Minimierungsproblem} we conclude that the sought minimizer satisfies
$\a_i = 0$ for all $i\in I\setminus J$ and therefore we find
\begin{align*}
M_W w  = \sum_{j\in J} \a_j  \sqrt{\lb_j} e_j = \sum_{j \in J} w_j e_j 
\end{align*}
with convergence in $W_X$. Since this shows $\snorm{M_Ww}_H \leq \snorm w_{\ell^2(J)}$
for all $w\in \ell^2(J)$ we conclude that $M_W$ can be uniquely extended 
to a bounded linear operator 
$M: \ell^2(J) \to H$, which is given by  
\begin{align*}
M w  = \sum_{j \in J} w_j e_j\, .
\end{align*}
\end{example} 
We note that Example \ref{example:6} covers the case of rotational invariant Gaussian random variables on a sphere $\mathbb{S}^n$, with $H=L^2(\mathbb{S}^n,\lambda)$ and $\lambda$ being a rotational invariant measure on $\mathbb{S}^n$.

\paragraph*{Acknowledgements}Daniel Winkle acknowledges funding by the International Max Planck Research School for Intelligent Systems (IMPRS-IS).

\bibliographystyle{plainnat}
\bibliography{daniel_refs}

\end{document}